\documentclass[a4paper, oneside]{amsart}

\usepackage[
paper=a4paper,
text={132mm,205mm},centering
]{geometry}
\usepackage{amssymb}
\usepackage[arrow]{xy}
\usepackage{pb-diagram,pb-xy}
\usepackage{graphicx,float}
\usepackage{hyperref}
\usepackage{xcolor,color}

\iftrue
\makeatletter
\def\@settitle{
  \vspace*{-0pt}
  \begin{flushleft}
    \LARGE\bfseries
    \strut\@title\strut
  \end{flushleft}
}
\def\@setauthors{
  \begingroup
  \def\thanks{\protect\thanks@warning}
  \trivlist
  \raggedright
  \large \@topsep27\p@\relax
  \advance\@topsep by -\baselineskip
  \item\relax
  \author@andify\authors
  \def\\{\protect\linebreak}
  \authors
  \ifx\@empty\contribs
  \else
    ,\penalty-3 \space \@setcontribs
    \@closetoccontribs
  \fi
  \normalfont
 \@setaddresses
  \endtrivlist
  \endgroup
}
\def\@setaddresses{\par
  \nobreak \begingroup
  \small
  \def\author##1{\nobreak\addvspace\smallskipamount}
  \def\\{\unskip, \ignorespaces}
  \interlinepenalty\@M
  \def\address##1##2{\begingroup
    \par\addvspace\bigskipamount\noindent
    \@ifnotempty{##1}{(\ignorespaces##1\unskip) }
    {\ignorespaces##2}\par\endgroup}
  \def\curraddr##1##2{\begingroup
    \@ifnotempty{##2}{\nobreak\noindent\curraddrname
      \@ifnotempty{##1}{, \ignorespaces##1\unskip}\/:\space
      ##2\par}\endgroup}
  \def\email##1##2{\begingroup
    \@ifnotempty{##2}{\nobreak\noindent E-mail address
      \@ifnotempty{##1}{, \ignorespaces##1\unskip}\/:\space
      \ttfamily##2\par}\endgroup}
  \def\urladdr##1##2{\begingroup
    \def~{\char`\~}
    \@ifnotempty{##2}{\nobreak\noindent\urladdrname
      \@ifnotempty{##1}{, \ignorespaces##1\unskip}\/:\space
      \ttfamily##2\par}\endgroup}
  \addresses
  \endgroup
  \global\let\addresses=\@empty
}
\def\@setabstracta{
    \ifvoid\abstractbox
  \else
    \skip@17pt \advance\skip@-\lastskip
    \advance\skip@-\baselineskip \vskip\skip@
    \box\abstractbox
    \prevdepth\z@ 
    \vskip-28pt
  \fi
}
\renewenvironment{abstract}{
  \ifx\maketitle\relax
    \ClassWarning{\@classname}{Abstract should precede
      \protect\maketitle\space in AMS document classes; reported}
  \fi
  \global\setbox\abstractbox=\vtop \bgroup
    \normalfont\small
    \list{}{\labelwidth\z@
      \leftmargin0pc \rightmargin\leftmargin
      \listparindent\normalparindent \itemindent\z@
      \parsep\z@ \@plus\p@
      
    }
    \item[\hskip\labelsep\bfseries\abstractname.]
}{
  \endlist\egroup
  \ifx\@setabstract\relax \@setabstracta \fi
}

\def\ps@headings{\ps@empty
  \def\@evenhead{
    \setTrue{runhead}
    \normalfont\scriptsize
    \rlap{\thepage}\hfill
    \def\thanks{\protect\thanks@warning}
    \leftmark{}{}}
  \def\@oddhead{
    \setTrue{runhead}
    \normalfont\scriptsize
    \def\thanks{\protect\thanks@warning}
    \rightmark{}{}\hfill \llap{\thepage}}
  \let\@mkboth\markboth
}\ps@headings

\def\section{\@startsection{section}{1}
  \z@{-1.4\linespacing\@plus-.5\linespacing}{.8\linespacing}
  {\normalfont\bfseries\Large}}
\def\subsection{\@startsection{subsection}{2}
  \z@{-.8\linespacing\@plus-.3\linespacing}{.5\linespacing\@plus.2\linespacing}
  {\normalfont\bfseries\large}}
\def\subsubsection{\@startsection{subsubsection}{3}
  \z@{.7\linespacing\@plus.2\linespacing}{-1.5ex}
  {\normalfont\bfseries}}
\def\@secnumfont{\bfseries}

\renewcommand\contentsnamefont{\bfseries}
\def\@starttoc#1#2{\begingroup
  \setTrue{#1}
  \par\removelastskip\vskip\z@skip
  \@startsection{}\@M\z@{\linespacing\@plus\linespacing}
    {.5\linespacing}{
      \contentsnamefont}{#2}
  \ifx\contentsname#2
  \else \addcontentsline{toc}{section}{#2}\fi
  \makeatletter
  \@input{\jobname.#1}
  \if@filesw
    \@xp\newwrite\csname tf@#1\endcsname
    \immediate\@xp\openout\csname tf@#1\endcsname \jobname.#1\relax
  \fi
  \global\@nobreakfalse \endgroup
  \addvspace{32\p@\@plus14\p@}
  \let\tableofcontents\relax
}
\def\contentsname{Contents}

\def\l@section{\@tocline{2}{.5ex}{0mm}{5pc}{}}
\def\l@subsection{\@tocline{2}{0pt}{2em}{5pc}{}}
\makeatother
\fi 

\theoremstyle{plain}
\newtheorem{theorem}{Theorem}[section]
\newtheorem{proposition}[theorem]{Proposition}
\newtheorem{corollary}[theorem]{Corollary}
\newtheorem{lemma}[theorem]{Lemma}

\theoremstyle{definition}
\newtheorem{definition}[theorem]{Definition}

\newtheorem{remark}[theorem]{Remark}

\makeatletter
\def\Nopagebreak{\@nobreaktrue\nopagebreak}
\makeatother

\def\Z{\mathbb{Z}}
\def\Q{\mathbb{Q}}
\def\R{\mathbb{R}}
\def\C{\mathbb{C}}
\def\N{\mathbb{N}}

\def\B{\mathcal{B}}
\def\K{\mathcal{K}}
\def\M{\mathcal{M}}

\def\P{\mathcal{P}}
\def\T{\mathcal{T}}
\def\S{\mathcal{S}}
\def\C{\mathcal{C}}
\def\H{\mathcal{H}}
\def\K{\mathcal{K}}
\def\hH{\widehat{\H}}
\def\hF{\widehat{F}}
\def\hG{\widehat{G}}
\def\hM{\widehat{M}}
\def\hN{\widehat{N}}
\def\hW{\widehat{W}}
\def\hMN{\widehat{M\cdot N}}
\def\l{\lambda}

\def\exp{\operatorname{exp}}
\def\Aut{\operatorname{Aut}}
\def\Ker{\operatorname{Ker}}
\def\Coker{\operatorname{Coker}}
\def\Im{\operatorname{Im}}

\def\sign{\operatorname{sign}}
\def\id{\mathrm{id}}
\def\max{\mathrm{max}}

\def\rank{\mathrm{rank}\;}

\def\zp{\mathbb{Z}_{(p)}}
\def\z2{\mathbb{Z}_{(2)}}

\def\F{\mathcal{F}}

\def\to{\mathchoice{\longrightarrow}{\rightarrow}{\rightarrow}{\rightarrow}}
\makeatletter
\newcommand{\shortxra}[2][]{\ext@arrow 0359\rightarrowfill@{#1}{#2}}
\def\longrightarrowfill@{\arrowfill@\relbar\relbar\longrightarrow}
\newcommand{\longxra}[2][]{\ext@arrow 0359\longrightarrowfill@{#1}{#2}}
\renewcommand{\xrightarrow}[2][]{\mathchoice{\longxra[#1]{#2}}%
  {\shortxra[#1]{#2}}{\shortxra[#1]{#2}}{\shortxra[#1]{#2}}}
\newcommand{\hooklongrightarrow}{\lhook\joinrel\longrightarrow}

\makeatother

\makeatletter
\def\Nopagebreak{\@nobreaktrue\nopagebreak}
\makeatother

\begin{document}

\title
[The homology cobordism group of homology cylinders]
{Invariants and structures of the homology cobordism group of homology cylinders}
\author{Minkyoung Song}
\address{
  Department of Mathematics\\
  POSTECH\\
  Pohang 790--784\\
  Republic of Korea
}
\email{pp1004@postech.ac.kr}

\begin{abstract}
The homology cobordism group of homology cylinders is a generalization of the mapping class group and the string link concordance group. We study this group and its filtrations by subgroups by developing new homomorphisms. First, we define extended Milnor invariants by combining the ideas of Milnor's link invariants and Johnson homomorphisms. They give rise to a descending filtration of the homology cobordism group of homology cylinders. We show that each successive quotient of the filtration is free abelian of finite rank. Second, we define Hirzebruch-type intersection form defect invariants obtained from iterated $p$-covers for homology cylinders. Using them, we show that the abelianization of the intersection of our filtration is of infinite rank. Also we investigate further structures in the homology cobordism group of homology cylinders which previously known invariants do not detect.
\end{abstract}

\maketitle

\setcounter{tocdepth}{2}
\tableofcontents

\section{Introduction}
The homology cobordism group of homology cylinders is an interesting object of study which extends the mapping class group and generalizes the concordance group of string links in homology 3-balls. The aim of this paper is to enhance our understanding of the structure of the group by developing two new invariants which are homomorphisms. We obtain the first invariant by combining the ideas of Milnor's link invariant and Johnson's homomorphism, and the second is a Hirzebruch-type intersection form defect from iterated $p$-covers. In this paper, manifolds are assumed to be compact and oriented. Our results hold in both topological and smooth categories. 

Let $\Sigma_{g,n}$ be a surface of genus $g$ with $n$ boundary components.
Roughly speaking, a \emph{homology cylinder} over $\Sigma_{g,n}$ is a homology cobordism between two copies of~$\Sigma_{g,n}$. A homology cylinder is endowed with two embeddings $i_+^{\vphantom{}}$ and $i_-^{\vphantom{}}$ of $\Sigma_{g,n}$ called \emph{markings}. The notion of homology cylinders was first introduced by Goussarov \cite{Go} and Habiro \cite{Ha} independently, as important model objects for their theory of finite type invariants of 3-manifolds which play the role of string links in the theory of finite type invariants of links. While the set $\C_{g,n}$ of marking-preserving homeomorphism types of homology cylinders is a monoid, the set $\H_{g,n}$ of homology cobordism classes becomes a group under juxtaposition. The group $\H_{g,n}$ was introduced by Garoufalidis and Levine as an enlargement of the mapping class group $\M_{g,n}$ of $\Sigma_{g,n}$~\cite{GL, L01}.
The group $\M_{g,n}$ injects into $\C_{g,n}$ and also into $\H_{g,n}$ (See \cite[p.~247]{L01}, \cite[Proposition 2.3]{CFK}).
Moreover, when $n>1$, $\H_{g,n}$ can be seen as a generalization of the concordance group of framed string links in homology balls.

In this paper, we assume that $n >0$, i.e.\ $\Sigma_{g,n}$ has nonempty boundary. We usually omit the subscripts and simply write $\Sigma$, $\C$, and $\H$ when $g,n$ are understood from the context. Let $I$ denote the interval~$[0,1]$. For a group $G$, $G_k$ denotes the $k$th term of lower central series given by $G_1=G$, $G_{k+1} = [G, G_k]$, and $G^{(k)}$ denotes the $k$th derived subgroup given by $G^{(0)}=G$, $G^{(k+1)}=[G^{(k)}, G^{(k)}]$.

In the literature, the structure of $\H$ was studied by constructing invariants. In particular, invariants which are group homomorphisms are essential in understanding the group structure.
Let $F= \pi_1(\Sigma)$ and $H=H_1(\Sigma)$.
In \cite{GL, L01}, Garoufalidis and Levine defined homomorphisms $\eta_q^{\vphantom{}}\colon \H_{g,1} \to \Aut(F/F_q)$ and a filtration $\H_{g,1}[q]:=\Ker \eta_q^{\vphantom{}}$ as extensions of the Johnson homomorphisms and the Johnson filtration of the mapping class group $\M_{g,1}$ \cite{Jo}. We call the maps $\eta_q^{\vphantom{}}$ on $\H_{g,1}$ \emph{Garoufalidis-Levine homomorphisms}. (In some literature, those on $\H_{g,1}$ are referred to also as `Johnson homomorphisms.')  Briefly, the invariants $\eta_q^{\vphantom{}}$ measure the difference between two markings on $F/F_q$. Garoufalidis and Levine determined the image of $\H_{g,1}[q]$ under~$\eta_q^{\vphantom{}}$ and showed each successive quotient $\H_{g,1}[q-1]/\H_{g,1}[q]$ is finitely generated free abelian. See also Remark~\ref{remark:GL} and the paragraph after Theorem~\ref{theorem:rank} for precise statements. We remark that the image of the Johnson subgroup of the mapping class group is unknown.
In \cite{M}, Morita obtained a homomorphism on $\H_{g,1}$ by taking the limit of a trace map composed with~$\eta_q^{\vphantom{}}$. He used this to show that the Torelli subgroup $\H_{g,1}[2]$ of $\H_{g,1}$ has infinite rank abelianization, while it is known that the Torelli subgroup of $\M_{g,1}$ is finitely generated for $g\geq 3$.
In \cite{S08, S}, Sakasai defined Magnus representations, which are crossed homomorphisms on $\H_{g,1}$ and homomorphisms on the subgroups~$\H_{g,1}[q]$. Using them, he proved that $\M_{g,1}$ is not a normal subgroup of $\H_{g,1}$ for $g\geq 3$.
Cha, Friedl, and Kim defined a torsion invariant in \cite{CFK}. They used it to show that the abelianization of $\H_{g,n}$ contains a direct summand isomorphic to $(\Z/2)^\infty$ if $b_1(\Sigma)>0$, and contains a direct summand isomorphic to $\Z^\infty$ if $n>1$.
In \cite{CHH}, Cochran, Harvey, and Horn considered signature invariants for $\H_{g,1}[2]$, which are the von Neumann-Cheeger-Gromov $L^2$-signature defects of bounding 4-manifolds. Their invariants are quasimorphisms on $\H_{g,1}[q]$ and send $\H_{g,1}[q]$ to a dense and infinitely generated subgroup of $\R$ for $g\geq 1$. They become homomorphisms on the kernel of Sakasai's Magnus representation on~$\H_{g,1}[q]$.

In fact, for $\H_{g,1}$, $\eta_q^{\vphantom{}}$ is related to the Milnor invariant of (string) links as described briefly below. The concordance group of $m$-component framed string links in homology balls is naturally identified, by taking the exterior, with~$\H_{0,m+1}$, and the total Milnor invariant of length $\leq q$ for string links can be viewed as a homomorphism $\mu_q^{\vphantom{}}$ defined on $\H_{0,m+1}$.
We denote its kernel by~$\H_{0,m+1}(q)$. Habegger established a bijection $\H_{0,2g+1}(2)\to \H_{g,1}[2]$, which is not a homomorphism but descends to an isomorphism between $\H_{0,2g+1}(q-1)/\H_{0,2g+1}(q)$ and $\H_{g,1}[q-1]/\H_{g,1}[q]$~\cite{Ha}. Levine found a monomorphism $\H_{0,g+1}\hookrightarrow \H_{g,1}$ which induces a monomorphism of $\H_{0,g+1}(q-1)/\H_{0,g+1}(q)$ into $\H_{g,1}[q-1]/\H_{g,1}[q]$~\cite{L01}. Habegger and Levine showed that $\mu_q^{\vphantom{}}$ and $\eta_q^{\vphantom{}}$ can be identified under these maps, respectively.

\subsection{Extended Milnor invariants}
\label{sec:1.1}
We will define a new invariant $\tilde\mu_q^{\vphantom{}}$ on $\H_{g,n}$ for arbitrary $(g,n)$ with $n \geq 1$. This is a common generalization of the Milnor $\bar\mu$-invariant and the Garoufalidis-Levine homomorphism which are defined only when $g=0$ and $n=1$, respectively.

As in \cite{HL}, string links have the advantage that their $\bar\mu$-invariants are well-defined without indeterminancy, in contrast to links, because a string link has well-defined (zero-linking) longitudes, as elements of the fundamental group of the exterior (see Section~\ref{sec:longitude} for details).
In fact, the Milnor invariant of a string link essentially represents the longitudes as elements of the free nilpotent quotient.

We generalize this to homology cylinders over a surface $\Sigma=\Sigma_{g,n}$ as follows.
Briefly speaking, we take $n-1$ fundamental group elements for a homology cylinder as analogs of the longitudes of a string link, 
and to extract more information from the fundamental group, 
we consider additional $2g$ elements that arise from symplectic basis curves of the surface~$\Sigma$. 
Note that $2g+n-1$, the total number of the elements we consider, is equal to the first Betti number~$b_1(\Sigma)$.
By taking the image of those elements in $F/F_q$, where $F=\pi_1(\Sigma)$, 
we define an extended Milnor invariant 
$$\tilde\mu_q^{\vphantom{}} \colon \H_{g,n}  \to (F/F_q)^{2g+n-1}. $$ 
 The precise definition is given in Section~\ref{sec:longitude}. It turns out that $\tilde\mu_q^{\vphantom{}}$ is equivalent to the Garoufalidis-Levine homomorphism $\eta_q^{\vphantom{}}$ for $n=1$, and to the Milnor $\bar\mu$-invariant for $g=0$.

We remark that the $2g+n-1$ elements used above are essential in understanding the fundamental group of the \emph{closure} of a homology cylinder (see Section~\ref{sec:definition} for the definition), from which we will extract signature defects and more generally Witt class invariants. These invariants will be discussed in the next subsection. 

We define a filtration by $\H(q) :=\Ker \tilde\mu_q^{\vphantom{}}$. This generalizes  the Garoufalidis-Levine's filtration $\H[q]$, in the sense that $\H(q) = \H[q]$ for $n=1$.
We remark that Garoufalidis-Levine's definition for $n=1$ can be applied to the case of $n>1$ to give a filtration, which we also denote by $\H[q]$; our filtration is finer than this, that is, we have $\H(q) \subset \H[q]$ in general.
The map $\tilde\mu_q^{\vphantom{}}$ is a crossed homomorphism. It is a homomorphism on both $\H[q]$ and~$\H(q-1)$. For more details, see Section~\ref{sec:product formula}.

Regarding the structure of the successive quotients of our filtration, we have the following result:
\begin{theorem}
  For each $q \geq 2$, the following hold: 
  \begin{enumerate}
	\item $\tilde\mu_{q}^{\vphantom{}}$ induces an injective homomorphism 
	$$\tilde\mu_{q}^{\vphantom{}}\colon \H(q-1)/\H(q) \hooklongrightarrow  (F_{q-1}/F_{q})^{2g+n-1}.$$
	Hence, $\H(q-1)/\H(q)$ is finitely generated free abelian.
	\item We have
	$$\max\{r_q(2g),r_q(g+n-1)\} \leq \rank \H(q-1)/\H(q) \leq r_q(2g+n-1)$$
		where $N_q(m)=\frac{1}{q}\sum_{d|q} \varphi(d) (m^{q/d})$, $\varphi$ is the M\"obius function, and $r_q(m)=m N_{q-1}(m)-N_q(m)$.
  \end{enumerate}
\end{theorem}
It is known that the injection in (1) is an isomorphism and the equality in (2) holds if either $g=0$, i.e.\ for string links \cite{Or}, or $n=1$~\cite{GL}.

Our next result is that a surface embedding gives rise to relationships between homology cobordism groups of homology cylinders over the surfaces and between their extended Milnor invariants as follows:
\begin{theorem}
  For any embedding $\imath\colon\Sigma_{g,n} \to \Sigma_{g',n'}$ with $n,n'\geq 1$, it induces a homomorphism $\tilde\imath\colon \H_{g,n} \to \H_{g',n'}$, and a function $f\colon (F/F_q)^{2g+n-1} \to (F'/F'_q)^{2g'+n'-1}$ which make the following diagram commute, where $F=\pi_1(\Sigma_{g,n})$ and $F'=\pi_1(\Sigma_{g',n'})$:
$$  \begin{diagram}
  	\node{\H_{g,n}} \arrow{e,t}{\tilde\imath} \arrow{s,r}{\tilde\mu_q^{\vphantom{}}} \node{\H_{g',n'}} \arrow{s,l}{\tilde\mu_q^{\vphantom{}}} \\
  	\node{(F/F_q)^{2g+n-1}} \arrow{e,t}{f} \node{(F'/F'_q)^{2g'+n'-1}}
  \end{diagram}$$
\end{theorem}
In addition, we present a sufficient condition for $f$ to be 1-1 and a sufficient condition for $\tilde\imath\colon \H_{g,n}\to\H_{g',n'}$ to be injective in Theorem~\ref{theorem:whole injection}. The former implies that the extended Milnor invariant of $\tilde\imath(M) \in \H_{g',n'}$ determines that of $M\in\H_{g,n}$. 

Applying the above result to appropriate surface embeddings, we obtain the following: 
\begin{corollary}
  \label{cor:whole injection}
  For any two pairs $(g,n)$ and $(g',n')$ satisfying $g \leq g'$ and $g+n \leq g'+n'$, there is an injective homomorphism
  $$\H_{g,n} \hooklongrightarrow \H_{g',n'} ,$$
which induces injections
  $$\H_{g,n}(q-1)/\H_{g,n}(q) \hooklongrightarrow \H_{g',n'}(q-1)/\H_{g',n'}(q)$$
for all $q \geq 2$.
\end{corollary}
Levine's monomorphism $\H_{0,g+1}\to \H_{g,n}$ \cite{L01} is a special case of this.

In the next subsection, we present our results on the structure of~$\H(\infty):=\bigcap_q \H(q)$.

\subsection{Hirzebruch-type intersection form defect invariants}

In \cite{C10}, Cha defined Hirzebruch-type intersection form defects for closed 3-manifolds.
In order to extract homology cobordism invariants, he considered towers of abelian $p$-covers. 
Let $d$ be a power of a prime~$p$. 
For a CW-complex $X$, a pair of a tower of iterated abelian $p$-covers and a homomorphism of the fundamental group of the top cover to $\Z_d$ is called a (\emph{$\Z_d$-valued}) \emph{$p$-structure} for~$X$ \cite{C09}.
(A precise definition is given in Section~\ref{sec:hirzebruch}.) 
We remark that any connected $p$-cover can be obtained as the top cover of a $p$-structure.
For a closed 3-manifold $M$ and a $p$-structure $\T$ for $M$ such that the top cover is zero in the bordism group $\Omega_3(B\Z_d)$, an invariant $\l(M,\T)$ is defined to be the difference between the Witt classes of the $\Q(\zeta_d)$-valued intersection form and the ordinary intersection form of a 4-manifold bounded by the top cover over $\Z_d$, where $\zeta_d=\exp(2\pi\sqrt{-1}/d)$. This lives in the Witt group $L^0(\Q(\zeta_d))$ of nonsingular hermitian forms over~$\Q(\zeta_d)$. 
This $\l(-,-)$ is a homology cobordism invariant in the sense that 
if $M$ and $N$ are homology cobordant, there is a 1-1 correspondence $\T_M \mapsto \T_N$ between $p$-structures for $M$ and $N$, $\l(M,\T_M)$ is defined if and only if $\l(N,\T_N)$ is, and in that case, $\l(M,\T_M)=\l(N,\T_N)$.

A map $f\colon X \to Y$ is called a \emph{$p$-tower map} if pullback gives rise to a 1-1 correspondence 
$$\Phi_f \colon \{p\textrm{-structures for } Y\} \to \{p\textrm{-structures for }X\}.$$
(For a more precise description, see Section~\ref{sec:defining condition}.) 
He found a sufficient condition for a map to be a $p$-tower map: for a map between CW-complexes with finite 2-skeletons, if it induces an isomorphism on the ``algebraic closures'' of their fundamental groups, then the map is a $p$-tower map \cite[Proposition~3.9]{C10}. (More information about the algebraic closure is given in Section~\ref{sec:hF-homology cylinder}.)
This applies to string links as follows. 
A string link $\sigma$ has canonical meridians, which give a meridian map $m\colon\bigvee S^1 \to M_\sigma$ where $M_\sigma$ is the surgery manifold of its closure. A string link is called an \emph{$\hF$-string link} if its meridian map induces an isomorphism on the algebraic closures of the fundamental groups. Using the above, for an $\hF$-string link $\sigma$, 
he defined a concordance invariant $\l_\T(\sigma):=\l(M_\sigma,\Phi_m^{-1}(\T))$ for each $p$-structure $\T$ for $\bigvee S^1$. He also proved that $\l_\T$ is a group homomorphism of the concordance group of $\hF$-string links \cite{C09}.

We apply the Hirzebruch-type invariants in \cite{C10,C09} to homology cylinders. Motivated by \cite{C09}, for homology cylinders, we define invariants parametrized by the $p$-structures for the base surface. For this purpose, we investigate when the composition $\hat{i}\colon \Sigma\xrightarrow{i_+} M \to \hM$, which we call the \emph{marking} for the closure $\hM$, is a $p$-tower map. As stated in the first part of the following theorem, the criterion is exactly the vanishing of our extended Milnor invariants. 

\begin{theorem}
  \leavevmode \Nopagebreak
  \begin{enumerate}
  	\item The marking $\hat{i} \colon \Sigma \to \hM$ is a $p$-tower map if and only if $\tilde\mu_q^{\vphantom{}}(M)$ vanishes for all $q$. In that case, $\l_\T(M):=\l(M, (\Phi_{\hat{i}})^{-1}(\T))$ is well-defined for any $p$-structure $\T$ for~$\Sigma$.
  	\item For any $p$-structure $\T$ for $\Sigma$,
    $$\l_\T \colon \H(\infty) \to L^0(\Q(\zeta_d))$$ 
  is a group homomorphism.
  \end{enumerate}
\end{theorem}

Furthermore, we give a generalization as follows. We consider certain special $p$-structures, which are called \emph{$p$-structures of order $q$}. Roughly, they are $p$-structures factoring through the $q$th lower central series quotient. (See Definition~\ref{definition:of order}).
We prove that the marking $\Sigma\to\hM$ induces a 1-1 correspondence between $p$-structures of order $q$ if and only if $\tilde\mu_q^{\vphantom{}}(M)$ vanishes. 
In this case, we define an invariant $\l_\T(M)$, with value in $\Z[\frac{1}{d}]\otimes_\Z L^0(\Q(\zeta_d))$. In other words, we define  
$$\l_\T \colon \H(q) \to \Z\Big[\frac{1}{d}\Big]\otimes_\Z L^0(\Q(\zeta_d))$$ for each $p$-structure $\T$ for $\Sigma$ of order~$q$.
However, it may not be a homomorphism. We present sufficient conditions on homology cylinders for $\l_\T$ to be additive in Theorem~\ref{theorem:additivity condition}.

We study the structure of $\H(\infty)$ using $\l_\T$ in Section~\ref{sec:effect}.
Performing infection by knots, we construct infinitely many homology cylinders with vanishing extended Milnor invariants, but distinguished by $\l_\T$.

\begin{theorem}
  \label{theorem}
  When $b_1(\Sigma)>1$,
  the abelianization of  $\H(\infty)$ contains $\Z^\infty$.
\end{theorem}

We define a boundary homology cylinder and an $\hF$-homology cylinder in Section~\ref{sec:boundary homology cylinder} and~\ref{sec:hF-homology cylinder} as analogs of the boundary (string) link and the $\hF$-(string) link, respectively. 
The above theorem also holds on the subgroup consisting of boundary homology cylinders and that consisting of $\hF$-homology cylinders. (See Theorem~\ref{theorem:infinite rank subgroups}.)

Moreover, our method detects homology cylinders that cannot be detected by other invariants we discussed
 before Section~\ref{sec:1.1} as follows:
\begin{theorem}
  Suppose $b_1(\Sigma)>1$. If $n=1$, the intersection of the kernels of Garoufalidis-Levine's homomorphisms $\eta_q$ \cite{GL}, Cha-Friedl-Kim's torsion invariant $\tau$ \cite{CFK}, the extended Milnor invariants $\tilde\mu_q$, Morita's homomorphism $\tilde{\rho}$ \cite{M}, Sakasai's Magnus representations $r_q$ \cite{S}, and Cochran-Harvey-Horn's signature invariants $\rho_q$ \cite{CHH} has infinite rank abelianization. If $n> 1$, the intersection of the kernels of $\eta_q$, $\tau$, $\tilde\mu_q$ (in this case, $\tilde{\rho}$, $r_q$, and $\rho_q$ are not defined) has infinite rank abelianization.
\end{theorem} 

\subsection{Solvable filtration of homology cylinders}
In Section~\ref{sec:cobordisms}, we investigate some other cobordisms of homology cylinders. 
Whitney towers and gropes play a key role in the study of topology of 4-manifolds and concordance of knots and links.
In \cite{COT}, Cochran, Orr, and Teichner introduced solvability of knots and related it to Whitney towers in 4-manifolds. Their filtrations on knots and links have been much studied as an approximation of sliceness. 
To study 3-manifolds with nonempty boundary, Cha defined Whitney tower cobordism and solvable cobordism of bordered 3-manifolds~\cite{C12}. Applying his definition to homology cylinders, it is straightforward to define the notion of $(r)$-solvable homology cylinders for $r\in \frac{1}{2}\Z_{\geq 0}$.
We show that $\l_\T$ can be used as obstructions to the solvability of homology cylinders:
\begin{theorem}
  Let $M \in \H(q)$ and $\T$ be a $p$-structure of height~$\leq h$ for $\Sigma$ of order~$q$. 
  If either \begin{enumerate}
	\item $M$ is $(h+1)$-solvable, or
	\item $M$ is $(h.5)$-solvable and satisfies one of \textnormal{(C1)--(C5)} of Theorem~\ref{theorem:additivity condition},
  \end{enumerate}	
  then $\l_\T(M)$ vanishes.
\end{theorem}
Here the height of a $p$-structure for $X$ is the height of the tower of iterated $p$-covers, a precise definition is given at the beginning of Section~\ref{sec:hirzebruch}.

We also refine Theorem~\ref{theorem}. Let $\F^{\H(\infty)}_{(r)}$ denote the subgroup consisting of $(r)$-solvable homology cylinders in $\H(\infty)$ for each $r \in \frac{1}{2}\Z_{\geq 0}$.

\begin{theorem}
  When $b_1(\Sigma) > 1$, the abelianization of ${\F^{\H(\infty)}_{(h)}}/{\F^{\H(\infty)}_{(h.5)}}$ is of infinite rank. 
\end{theorem}

We remark that the analogs hold for the solvable filtration of the subgroups of boundary homology cylinders and that of $\hF$-homology cylinders. See Theorem~\ref{theorem:solvable filtration}.

The paper is organized as follows. 
In Section~\ref{sec:definition}, we recall basic definitions and examples of homology cylinders and their homology cobordism groups.
In Section~\ref{sec:representation}, we define extended Milnor invariants on $\H_{g,n}$.
In Section~\ref{sec:subgroups}, we study the filtration $\H(q)$ associated to the extended Milnor invariants and subgroups consisting of boundary homology cylinders, $\hF$-homology cylinders.
In Section~\ref{sec:hirzebruch}, we define Hirzebruch-type invariants of homology cylinders and give sufficient conditions for additivity of the invariants.
In Section~\ref{sec:effect}, by investigating the effect of infection, we detect a rich structure of $\H$ which has not been detected previously. 
Finally in Section~\ref{sec:cobordisms}, we study nilpotent cobordism and solvable filtrations of homology cylinders using our invariants.

\subsubsection*{Acknowledgements}
The author thanks her advisor Jae Choon Cha for his advice and guidance.
This research was partially supported by NRF grants 2013067043 and 2013053914.

\section{Homology cylinders and their homology cobordism groups}
  \label{sec:definition}
We recall precise definitions about homology cylinders. Let $\Sigma=\Sigma_{g,n}$ be a surface of $g$ genus with $n$ boundary components.
\begin{definition}
  A \emph{homology cylinder over} $\Sigma$ consists of a 3-manifold $M$ with two embeddings $i_+^{\vphantom{}},~i_-^{\vphantom{}}\colon \Sigma \hookrightarrow \partial M$, called \emph{markings}, such that
  \begin{enumerate}
    \item   $i_+^{\vphantom{}}|_{\partial \Sigma} = i_-^{\vphantom{}}|_{\partial \Sigma}$,
    \item   $i_+\cup i_- \colon \Sigma\cup_\partial (-\Sigma) \to \partial M $ is an orientation-preserving homeomorphism, and
    \item   $i_+^{\vphantom{}}, i_-^{\vphantom{}}$ induce isomorphisms $H_*(\Sigma;\Z)\to H_*(M;\Z)$.
  \end{enumerate}
We denote a homology cylinder by $(M,i_+^{\vphantom{}},i_-^{\vphantom{}})$ or simply by~$M$.
\end{definition}

Two homology cylinders $(M,i_+^{\vphantom{}},i_-^{\vphantom{}})$ and $(N,j_+^{\vphantom{}},j_-^{\vphantom{}})$ over $\Sigma_{g,n}$ are said to be \emph{isomorphic} if there exists an orientation-preserving homeomorphism $f\colon M \to N$ satisfying $j_+^{\vphantom{}}=f \circ i_+^{\vphantom{}}$ and $j_-^{\vphantom{}}=f \circ i_-^{\vphantom{}}$. Denote by $\C_{g,n}$ the set of all isomorphism classes of homology cylinders over~$\Sigma_{g,n}$. We define a product operation on $\C_{g,n}$ by
$$(M,i_+^{\vphantom{}},i_-^{\vphantom{}})\cdot (N,j_+^{\vphantom{}},j_-^{\vphantom{}}):=(M\cup_{i_-^{\vphantom{}}\circ (j_+^{\vphantom{}})^{-1}} N, i_+^{\vphantom{}}, j_-^{\vphantom{}})$$
for $(M,i_+^{\vphantom{}},i_-^{\vphantom{}}),~(N,j_+^{\vphantom{}},j_-^{\vphantom{}}) \in \C_{g,n}$, which endows $\C_{g,n}$ with a monoid structure. The identity is $(\Sigma_{g,n} \times I/(z,0)=(z,t)~ (z\in\partial\Sigma, t\in I), \id \times 1, \id\times 0)$. For later use, we denote this trivial homology cylinder by~$E$.

\begin{definition}
  Two homology cylinders $(M,i_+^{\vphantom{}},i_-^{\vphantom{}})$ and $(N,j_+^{\vphantom{}},j_-^{\vphantom{}})$ over $\Sigma_{g,n}$ are said to be \emph{homology cobordant} if there exists a 4-manifold $W$ such that
  \begin{enumerate}
    \item   $\partial W = M \cup (-N) /\sim$, where $\sim$ identifies $i_+^{\vphantom{}}(x)$ with $j_+^{\vphantom{}}(x)$ and $i_-^{\vphantom{}}(x)$ with $j_-^{\vphantom{}}(x)$ for all $x\in\Sigma_{g,n}$, and
    \item   the inclusions $M \hookrightarrow W$, $N \hookrightarrow W$ induce isomorphisms on the integral homology.
  \end{enumerate}
\end{definition}
We denote by $\H_{g,n}$ the set of homology cobordism classes of elements of~$\C_{g,n}$. By abuse of notation, we also write $M$ for the class of~$M$. The monoid structure on $\C_{g,n}$ descends to a group structure on~$\H_{g,n}$, with $(M,i_+^{\vphantom{}},i_-^{\vphantom{}})^{-1}=(-M,i_-^{\vphantom{}},i_+)$. We call this group the \emph{homology cobordism group} of homology cylinders. Actually, there are two kinds of groups $\H_{g,n}^{\mathrm{smooth}}$ and $\H_{g,n}^{\mathrm{top}}$ depending on whether the homology cobordism is smooth or topological, and there exists a canonical epimorphism $\H_{g,n}^{\mathrm{smooth}} \twoheadrightarrow \H_{g,n}^{\mathrm{top}}$ whose kernel contains an abelian group of infinite rank~\cite{CFK}. In this paper, however, the author does not distinguish the two cases since everything holds in both cases.

Both $\H_{0,0}$ and $\H_{0,1}$ are isomorphic to the group of homology cobordism classes of integral homology 3-spheres. The group $\H_{0,2}$ is isomorphic to the concordance group of framed knots in homology 3-spheres. For $n\geq 3$, $\H_{0,n}$ is isomorphic to the concordance group of framed ($n-1$)-component string links in homology 3-balls, or equivalently, in homology cylinders over $D^2=\Sigma_{0,1}$. Similarly, $\H_{g,n}$ can be considered to be the concordance group of framed ($n-1$)-component string links in homology cylinders over~$\Sigma_{g,1}$. The fact that the mapping class group over $\Sigma_{g,n}$ is a subgroup of $\H_{g,n}$ implies $\H_{g,n}$ is non-abelian except $(g,n)=(0,0), (0,1)$ and~$(0,2)$.

For any homology cylinder $M$, there is an associated closed manifold $\hM$ obtained from $M$ by identifying $i_+^{\vphantom{}}(z)$ and~$i_-^{\vphantom{}}(z)$ for each $z\in \Sigma$. We call it the \emph{closure} of~$M$. When $M$ is considered as an exterior of a framed string link, $\hM$ is just the surgery manifold of the closure of the string link. Both $i_+^{\vphantom{}}, i_-^{\vphantom{}}$ composed with the quotient map give an embedding $\hat{i}\colon \Sigma\to\hM$, which we call the \emph{marking} for~$\hM$.

\section{Generalization of Milnor invariants and Garoufalidis-Levine homomorphisms}
  \label{sec:representation}
Let $\partial_1^{\vphantom{}},\partial_2^{\vphantom{}},\ldots,\partial_n^{\vphantom{}}$ be the boundary components of~$\Sigma$. Choose a basepoint $*$ of $\Sigma$ on $\partial_n^{\vphantom{}}$ and fix a generating set $\{x_1^{\vphantom{}},\ldots,x_{n-1}^{\vphantom{}},m_1^{\vphantom{}},\ldots,m_g^{\vphantom{}},l_1^{\vphantom{}},\ldots,l_g^{\vphantom{}}\}$ for $\pi_1(\Sigma,*)$ as in Figure~\ref{figure:Sigma} such that $x_i^{\vphantom{}}$ is homotopic to the $i$th boundary component $\partial_i^{\vphantom{}}$ and $m_j^{\vphantom{}}$, $l_j^{\vphantom{}}$ correspond to a meridian and a longitude of the $j$th handle. Since our $n$ is nonzero, the group is free on the above $2g + n -1$ generators. Let $F=\pi_1(\Sigma,*).$ For the generators in Figure~\ref{figure:Sigma}, the element $[\partial_n^{\vphantom{}}]\in\pi_1(\Sigma,*)$ is represented by $\prod_i x_i^{\vphantom{}}\prod_j[m_j^{\vphantom{}},l_j^{\vphantom{}}]$. We will use this later, to prove Theorem~\ref{theorem:rank}.

\begin{figure}[h]
  \begin{center}
    \includegraphics[scale=.9]{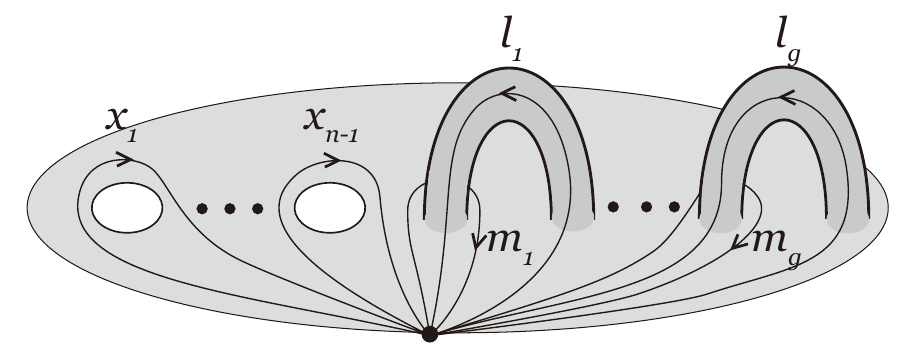}
    \caption{A generating set for $\pi_1(\Sigma_{g,n})$}
    \label{figure:Sigma}
  \end{center}
\end{figure}

\subsection{Extended Milnor invariants on $\H_{g,n}$}
  \label{sec:longitude}
First, we define Milnor invariants of homology cylinders similarly to those of string links. Let $(M, i_+^{\vphantom{}},i_-^{\vphantom{}})$ be a homology cylinder over~$\Sigma$. The chosen $x_i^{\vphantom{}}$ is of the form $[\alpha_i^{\vphantom{}}\cdot\beta_i^{\vphantom{}}\cdot \alpha_i^{-1}]$ for a closed path $\beta_i^{\vphantom{}}\colon I \to I/\partial I\xrightarrow{\simeq} \partial_i^{\vphantom{}}$ such that the latter map is a homeomorphism and a path $\alpha_i^{\vphantom{}}$ from $*$ to $\beta_i^{\vphantom{}}(0)$. The orientation of $\beta_i^{\vphantom{}}$ is determined by $x_i^{\vphantom{}}$.
Consider the loop $(i_+^{\vphantom{}}\circ\alpha_i^{\vphantom{}})\cdot(i_-^{\vphantom{}}\circ\alpha_i^{-1})$ in $M$. If $M$ were a framed string link exterior, this loop would be its $i$th longitude. We define $\l_i^{\vphantom{}}$ to be the class of the loop $(i_+^{\vphantom{}}\circ\alpha_i^{\vphantom{}})\cdot(i_-^{\vphantom{}}\circ\alpha_i^{-1})$ in $\pi_1(M,i_+^{\vphantom{}}(*))$. It is independent of the choice of $\alpha_i^{\vphantom{}}$, and it depends only on the choice of $x_i^{\vphantom{}}$ in~$\pi_1(\Sigma,*)$. We will show this at the end of this subsection. By Stallings' theorem~\cite{St}, $i_+^{\vphantom{}}$ induces an isomorphism
$$F/F_q=\pi_1(\Sigma) / \pi_1(\Sigma)_q  \xrightarrow[(i_+^{\vphantom{}})_{*q}^{\vphantom{}}]{\cong} \pi_1(M) / \pi_1(M)_q$$ for every $q\in\mathbb{N}$. We define $\mu_q^{\vphantom{}}(M)_i^{\vphantom{}}$ to be the inverse image of $\l_i^{\vphantom{}}$ in $F/F_q$ and $\mu_q^{\vphantom{}}(M)$ to be the ($n-1$)-tuple $(\mu_q^{\vphantom{}}(M)_1^{\vphantom{}},\ldots,\mu_q^{\vphantom{}}(M)_{n-1}^{\vphantom{}})$ $\in$ $(F/F_q)^{n-1}$. Also, $\mu(M)$ can be defined to be $((\mu_q^{\vphantom{}}(M)_1^{\vphantom{}})_{q\in\N}^{\vphantom{}},\ldots, (\mu_q^{\vphantom{}}(M)_{n-1}^{\vphantom{}})_{q\in\N}^{\vphantom{}})$ as an element of~$(\varprojlim_q F/F_q)^{n-1}$. 
For $(i_\pm^{\vphantom{}})_*\colon F\to \pi_1(M)$ induced by markings $i_\pm^{\vphantom{}}$, the element $\l_i^{\vphantom{}}$ indicates the difference between $(i_+^{\vphantom{}})_*(x_i^{\vphantom{}})$ and $(i_-^{\vphantom{}})_*(x_i^{\vphantom{}})$ in $\pi_1(M)$ as follows:
$$ (i_-^{\vphantom{}})_*(x_i^{\vphantom{}}) = \l_i^{-1}\cdot (i_+^{\vphantom{}})_*(x_i^{\vphantom{}}) \cdot \l_i^{\vphantom{}}.$$
Also $\mu_q^{\vphantom{}}(M)_i^{\vphantom{}}$ indicates the difference of two markings on $x_i^{\vphantom{}}$ in~$F/F_q$:
\begin{equation}
  ((i_+^{\vphantom{}})_{*q}^{-1}\circ(i_-^{\vphantom{}})_{*q}^{\vphantom{}})(x_i^{\vphantom{}})=\mu_q^{\vphantom{}}(M)_i^{-1}\cdot x_i^{\vphantom{}}\cdot   \mu_q^{\vphantom{}}(M)_i^{\vphantom{}}.
\end{equation}
For the case of $g=0$, the invariant $\mu_q^{\vphantom{}}$ is equivalent to all the Milnor $\bar\mu$-invariants of length $\leq q$ for framed string links.

Now we extend this for the remaining generators $m_j^{\vphantom{}}$ and $l_j^{\vphantom{}}$ of~$F$. Denote $\l'_j=(i_+^{\vphantom{}})_*(m_j^{\vphantom{}})\cdot (i_-^{\vphantom{}})_*(m_j^{-1})$ and  $\l''_j=(i_+^{\vphantom{}})_*(l_j^{\vphantom{}})\cdot (i_-^{\vphantom{}})_*(l_j^{-1})$. Let $\mu'_q(M)_j^{\vphantom{}}$ and $\mu''_q(M)_j^{\vphantom{}}$ be the inverse images of $\l'_j$ and $\l''_j$ under the above isomorphism $(i_+^{\vphantom{}})_{*q}^{\vphantom{}}$, respectively, for $j=1,\ldots,g.$ Then, clearly 
$$(i_-^{\vphantom{}})_*(m_j^{\vphantom{}})=\l_j'^{-1} \cdot (i_+^{\vphantom{}})_*(m_j^{\vphantom{}})\;, \quad  (i_-^{\vphantom{}})_*(l_j^{\vphantom{}})=\l_j''^{-1}\cdot (i_+^{\vphantom{}})_*(l_j^{\vphantom{}})\quad  \textrm{  in } \pi_1(M),$$
and hence
\begin{align}
((i_+^{\vphantom{}})_{*q}^{-1}\circ(i_-^{\vphantom{}})_{*q}^{\vphantom{}})(m_j^{\vphantom{}})&=\mu_q'(M)_j^{-1}\cdot m_j^{\vphantom{}} , \\ ((i_+^{\vphantom{}})_{*q}^{-1}\circ(i_-^{\vphantom{}})_{*q}^{\vphantom{}})(l_j^{\vphantom{}})&=\mu_q''(M)_j^{-1}\cdot l_j^{\vphantom{}}\quad\textrm{  in } F/F_q . \nonumber
\end{align}

We denote by $\tilde\l(M)$ or simply by $\tilde{\l}$ the ($2g+n-1$)-tuple of $\l_i^{\vphantom{}}, \l_j'$, and $\l_j''$ of $\pi_1(M)$ and by $\tilde\mu_q^{\vphantom{}}(M)$ the ($2g+n-1$)-tuple of $\mu_q^{\vphantom{}}(M)_i^{\vphantom{}},\mu_q'(M)_j^{\vphantom{}}$, and $\mu''_q(M)_j^{\vphantom{}}$ of~$F/F_q$. the total $\tilde\mu(M)$ is defined to be $((\tilde\mu_q^{\vphantom{}}(M)_1^{\vphantom{}})_{q\in\N}^{\vphantom{}},\ldots, (\tilde\mu_q^{\vphantom{}}(M)_{2g+n-1}^{\vphantom{}})_{q\in\N}^{\vphantom{}})$ as an element of~$(\varprojlim_q F/F_q)^{2g+n-1}$. 

Remark that $\tilde\l(M)$ plays an important role in studying the fundamental group of~$\hM$. 
More precisely, $\pi_1(\hM)$ can be obtained from $\pi_1(M)$ by killing all~$\tilde\l(M)_k^{\vphantom{}}$. Hence, $\tilde\mu_q^{\vphantom{}}(M)$ vanishes if and only if $\hat{i}$ induces an isomorphism $F/F_q\cong \pi_1(\hM)/\pi_1(\hM)_q$. It is similar to case of (string) links: for a (string) link $L$, the Milnor invariants of length$\leq q$ vanish if and only if a meridian map $\bigvee S^1\to M_L$ induces an isomorphism on $\pi_1(-)/\pi_1(-)_q$ where $M_L$ is the surgery manifold (of the closure).  
In this sense, $\tilde\mu$ is a more appropriate analog of Milnor's $\bar\mu$-invariants of string links, compared with~$\mu$. 

\begin{theorem}
  For any $q$, $\tilde\mu_q^{\vphantom{}}$ is a homology cobordism invariant.
  In other words, we have
  $$\tilde\mu_q^{\vphantom{}}\colon \H_{g,n}\to (F/F_q)^{2g+n-1}.$$
\end{theorem}
\begin{proof}
  Suppose $W$ is a homology cobordism between homology cylinders $(M, i_+^{\vphantom{}}, i_-^{\vphantom{}})$ and~$(M',i_+',i_-')$.
  The diagram below is commutative where the right maps are induced by natural inclusions. In $W$, $i_+^{\vphantom{}}(z)$ and $i_+'(z)$ are identified, and $i_-^{\vphantom{}}(z)$ and $i_-'(z)$ are identified for each $z\in \Sigma$. Hence $\tilde\l(M)_k^{\vphantom{}} \in \pi_1(M)$ and $\tilde\l(M')_k^{\vphantom{}} \in\pi_1(M')$ correspond to the same element of~$\pi_1(W)$. 
  $$\begin{diagram} \dgARROWLENGTH=1.5em
  	\node[2]{\pi_1(M)}	\arrow{se} \\
  	\node{F} \arrow{ne,l}{(i_+^{\vphantom{}})_*} \arrow{se,r}{(i_+')_*}		\node[2]{\pi_1(W)} \\
  	\node[2]{\pi_1(M')}	\arrow{ne}
  \end{diagram}$$
   Since all the homomorphisms of the diagram induce isomorphisms on~$\pi_1(-)/\pi_1(-)_q$, $\tilde\mu_q^{\vphantom{}}(M)= \tilde\mu_q^{\vphantom{}}(M')$ in $F/F_q$.
\end{proof}

Now, we investigate the well-definedness of $\tilde\mu_q^{\vphantom{}}$ and the effect of change of generating sets.

\begin{proposition}
  The invariant $\tilde\mu_q^{\vphantom{}}$ is independent of the the choice of $\{\alpha_i^{\vphantom{}}\}$ and depends only on (the basepoint $*$ of $\Sigma$ and) the (ordered) generating set $\{x_i^{\vphantom{}},m_j^{\vphantom{}},l_j\}_{i<n,j\leq g}$ of~$\pi_1(\Sigma,*)$. 
\end{proposition}
\begin{proof}
Let $x_i^{\vphantom{}}=[\alpha_i^{\vphantom{}}\cdot \beta_i^{\vphantom{}} \cdot \alpha_i^{-1}] = [\alpha_i' \cdot \beta_i' \cdot \alpha_i'^{-1}]$ for paths $\beta_i^{\vphantom{}},\beta_i'^{\vphantom{}}\colon I\to I/\partial I \xrightarrow{\simeq} \partial_i^{\vphantom{}}$ and $\alpha_i^{\vphantom{}}, \alpha_i'$ from $\ast$ to $\beta_i^{\vphantom{}}(0), \beta_i'(0)$, respectively.   
We may assume $\beta_i^{\vphantom{}}=\beta_i'$ since $\l(-)_i^{\vphantom{}}$ is unchanged under connecting a path in $\partial_i^{\vphantom{}}$ to~$\alpha_i^{\vphantom{}}$. 
The loop $(\alpha_i^{\vphantom{}}\cdot\beta_i^{\vphantom{}}\cdot\alpha_i^{-1})
  (\alpha_i'\cdot\beta_i^{-1}\cdot\alpha_i'^{-1})$ is a null-homotopic, and $\big[[\alpha_i'^{-1}\cdot\alpha_i^{\vphantom{}}], [\beta_i^{\vphantom{}}]\big] = 1 $ in the free group $\pi_1(\Sigma,\beta_i^{\vphantom{}}(0))$. Thus $[\alpha_i'^{-1}\cdot\alpha_i^{\vphantom{}}]=[\beta_i^{\vphantom{}}]^k$ for some~$k$. 
   Therefore $\alpha_i^{\vphantom{}}$ and $\alpha_i'$ determine the same $\l(M)_i^{\vphantom{}} \in \pi_1(M)$ for each homology cylinder~$M$.
\end{proof}

\begin{proposition}
 \label{proposition:effect of change}
  Let $A=\{x_i^{\vphantom{}},m_j^{\vphantom{}},l_j^{\vphantom{}}\}_{i<n,j\leq g}$  and $B=\{x_i',m_j',l_j'\}_{i<n,j\leq g}$ be generating sets of $\pi_1(\Sigma,\ast)=:F$ and $\pi_1(\Sigma,\ast')=:F'$, respectively, such that $x_i^{\vphantom{}}$ and $x_i'$ are homotopic to a boundary component. Suppose $\tilde\mu_q^A ,\tilde\mu_q^B\colon \H_{g,n}\to (F/F_q)^{2g+n-1}$ are the extended Milnor invariants with respect to $A$ and~$B$, respectively.
  Then there exists a bijection $f\colon (F/F_q)^{2g+n-1}\to (F/F_q)^{2g+n-1}$ which makes the following diagram commute:
  $$\begin{diagram}  
  	\node[2]{\H_{g,n}} \arrow{sw,l}{\tilde\mu_q^A} \arrow{se,l}{\tilde\mu_q^B} \\
  	\node{(F/F_q)^{2g+n-1}} \arrow[2]{e,b}{f}	\node[2]{(F'/F'_q)^{2g+n-1}}
  \end{diagram}$$
  
\end{proposition}
\begin{proof}
Let $$z=(z_1^{\vphantom{}},\ldots,z_{n-1}^{\vphantom{}};z_1',\ldots,z_g';z_1'',\ldots,z_g'') \in (F/F_q)^{2g+n-1}.$$
  We can assume $\ast=\ast'$ by the following claim: \\
  \textbf{Claim.} Suppose $\ast\neq\ast'$ and $\{x_i^{\vphantom{}}, m_j^{\vphantom{}}, l_j^{\vphantom{}}\}$ be a generating set of~$\pi_1(\Sigma,*)$. Then there are a generating set $\{x_i',m_i',l_j'\}$ of $\pi_1(\Sigma,*')$ and a bijection $f$ making the above diagram commute.
  
  To prove the claim, we consider two cases:
  \begin{enumerate}
  	\item[Case 1.] $*$ and $*'$ are in the same boundary component. \\
  		Choose a path $\gamma$ from $\ast'$ to $\ast$ in the boundary component, and let $\{x_i',m_j',l_j'\}$ be the generating set of $\pi_1(\Sigma,\ast')$ induced by $\gamma$-conjugation. Then $\tilde\mu_q^{\vphantom{}}$ is unchanged, i.e.\ $f=\id$. 
  	\item[Case 2.] $\ast$ and $\ast'$ are in different components. \\
  		We may assume $\ast'\in\partial_1$. Then  $x_1=[\alpha_1^{\vphantom{}}\cdot\beta_1^{\vphantom{}}\cdot\alpha_1^{-1}]$ for some paths $\alpha_1^{\vphantom{}}$ from $\ast$ to $\ast'$ and $\beta_1^{\vphantom{}}\colon I\to I/\partial I\xrightarrow{\simeq}\partial_1^{\vphantom{}}$ with $\beta_1^{\vphantom{}}(0)=\ast'$. Let $\gamma$ be~$\alpha_1^{-1}$. For the isomorphism $\phi\colon \pi_1(\Sigma,*)\to \pi_1(\Sigma,*')$ induced by $\gamma$-conjugation, let $x_1'=\phi([\partial_n^{\vphantom{}}])$, $x_i'=\phi(x_i^{\vphantom{}})$ for $i=2,\ldots, n-1$ and $m_j'=\phi(m_j^{\vphantom{}})$, $l_j'=\phi(l_j^{\vphantom{}})$ for $j=1,\ldots,g$.
  		Define a function $f \colon (F/F_q)^{2g+n-1}\to (F'/F'_q)^{2g+n-1}$ by
$$f(z)=\phi\big(z_1^{\vphantom{}},z_2^{\vphantom{}}z_1^{-1},\ldots,z_{n-1}^{\vphantom{}}z_1^{-1}; (m_j^{\vphantom{}}z_1^{\vphantom{}}m_j^{-1}z_j' z_1^{-1})_{j\leq g}; (l_j^{\vphantom{}}z_1^{\vphantom{}}l_j^{-1}z_j''z_1^{-1})_{j\leq g}\big),$$
  	then it gives the commutative diagram. 
   \end{enumerate}
   Now we prove the proposition with the same basepoint~$*=*'$ and $F=F'$.
   By reordering, we may assume that $x_i^{\vphantom{}}$ is homotopic to~$(x_i')^{\pm 1}$.
   We can choose paths $\alpha_i^{\vphantom{}}$ and $\alpha_i'$ with the same endpoints such that $x_i^{\vphantom{}}=[\alpha_i^{\vphantom{}}\cdot \beta_i^{\vphantom{}}\cdot \alpha_i^{-1}]$ and $x_i'=[\alpha_i'\cdot \beta_i^{\pm} \cdot\alpha_i'^{-1}]$ for $\beta_i^{\vphantom{}}\colon I\to I/\partial I \xrightarrow{\simeq}\partial_i^{\vphantom{}}$.
   Let $\gamma_i^{\vphantom{}}=[\alpha_i^{\vphantom{}}\cdot \alpha_i'^{-1}]$, $\gamma_j'=m_j^{\vphantom{}} m_j'^{-1}$ and $\gamma_j''=l_j^{\vphantom{}} l_j'^{-1}$,
   and let  
\begin{align*} 
  \gamma_i^{\vphantom{}}&= \omega_i^{\vphantom{}}(x_1^{\vphantom{}},\ldots,x_{n-1}^{\vphantom{}};m_1^{\vphantom{}},\ldots,m_g^{\vphantom{}};l_1^{\vphantom{}},\ldots,l_g^{\vphantom{}}),\\   
  \gamma'_j&=\omega_j'(x_1^{\vphantom{}},\ldots,x_{n-1}^{\vphantom{}};m_1^{\vphantom{}},\ldots,m_g^{\vphantom{}};l_1^{\vphantom{}},\ldots,l_g^{\vphantom{}}), \\ 
  \gamma''_j&=\omega_j''(x_1^{\vphantom{}},\ldots,x_{n-1}^{\vphantom{}};m_1^{\vphantom{}},\ldots,m_g^{\vphantom{}};l_1^{\vphantom{}},\ldots,l_g^{\vphantom{}})
\end{align*}
be the words in $x_i^{\vphantom{}},m_j^{\vphantom{}},l_j^{\vphantom{}}~(i<n,j\leq g)$.
Suppose $\varphi\colon (F/F_q)^{2g+n-1}\to (F/F_q)^{2g+n-1}$ is a function defined by
$$ z_i^{\vphantom{}}\mapsto z_i^{-1} x_i^{\vphantom{}} z_i^{\vphantom{}};\quad
 z_j' \mapsto  z_j'^{-1} m_j^{\vphantom{}} ;\quad
z_j''\mapsto  z_j''^{-1} l_j^{\vphantom{}}.$$
Define $f\colon (F/F_q)^{2g+n-1} \to (F/F_q)^{2g+n-1}$ by 
$$f(z)=\Big(\big(\gamma_i^{-1} z_i^{\vphantom{}} \omega_i^{\vphantom{}}(\varphi(z)))\big)_{i< n};\big(\gamma_j'^{-1} z_j' \omega_j'(\varphi(z))\big)_{j\leq g};  \big(\gamma_j''^{-1} z_j'' \omega_j''(\varphi(z))\big)_{j\leq g}\Big).$$
The verification that the diagram commutes is left to the reader.
A similar construction gives the inverse of $f$. It follows that $f$ is bijective.
\end{proof}
Therefore, $\tilde\mu^A_q(M)$ determines $\tilde\mu^B_q(M)$ and vice versa.
 
\subsection{Garoufalidis-Levine homomorphisms on $\H_{g,n}$}

Garoufalidis and Levine defined the homomorphism $\eta_q^{\vphantom{}} \colon \H_{g,1} \to \Aut(F/F_q)$ to be $(i_+^{\vphantom{}})_{*q}^{-1}\circ (i_-^{\vphantom{}})_{*q}^{\vphantom{}}$ where $(i_\pm^{\vphantom{}})_{*q}\colon F/F_q\xrightarrow{\cong} \pi_1(M)/\pi_1(M)_q$ as before. In the same way, a homomorphism $\eta_q^{\vphantom{}} \colon \H_{g,n} \to \Aut(F/F_q)$ can be defined.

We compare $\eta_q^{\vphantom{}}$ and $\tilde\mu_q^{\vphantom{}}$.
From (1) and (2) in Section~\ref{sec:longitude}, 
\begin{equation*}\eta_q^{\vphantom{}}(-)(x_i^{\vphantom{}})=\mu_{q-1}^{\vphantom{}}(-)_i^{-1} x_i^{\vphantom{}} \mu_{q-1}^{\vphantom{}}(-)_i^{\vphantom{}}, \end{equation*}
 $$\eta_q^{\vphantom{}}(-)(m_j^{\vphantom{}}) = \mu_q'(-)_j^{-1}m_j^{\vphantom{}}\quad\textrm{and}\quad\eta_q^{\vphantom{}}(-)(l_j^{\vphantom{}}) = \mu_q''(-)_j^{-1}l_j^{\vphantom{}}.$$
Hence, $(\mu_{q-1}^{\vphantom{}}, \mu_q', \mu_q'')$  determines~$\eta_q^{\vphantom{}}$, but the converse does not hold. For example, $\mu_{q-1}^{\vphantom{}}(M)_i^{\vphantom{}}$ can be the class of~$x_i^k$ even though $\eta_q^{\vphantom{}}(M)=\id$. In $\H_{g,1}$ or in $\Ker\mu_2^{\vphantom{}}\subset \H_{g,n}$, the triple $(\mu_{q-1}^{\vphantom{}}, \mu_q',\mu_q'')$ is equivalent to $\eta_q^{\vphantom{}}$, by the following lemma: 
\begin{lemma}
  \label{lemma:MKS}
      For a homology cylinder $M$, 
	if $\mu_2^{\vphantom{}}(M)_i^{\vphantom{}}=1$ and $[x_i^{\vphantom{}},\mu_{q-1}^{\vphantom{}}(M)_i^{\vphantom{}}]=1$ in $F/F_q$, then $\mu_{q-1}^{\vphantom{}}(M)_i^{\vphantom{}}=1$.
\end{lemma}
\begin{proof}
We will prove  that if $a\in F_2$ and $[x_i^{\vphantom{}},a]\in F_q$ then $a\in F_{q-1}$.
Consider the Magnus expansion of $F$ into the algebra of formal power series in noncommutative $2g+n-1$ variables $X_1,\ldots, X_{2g+n-1}$
$$\mathcal{M}\colon F\to Z[[X_1,\ldots,X_{2g+n-1}]]$$
which sends the $k$-th generator of $F$ to $1+X_k$. We order the generators so that $\mathcal{M}(x_i^{\vphantom{}})=1+X_i$.
It is well-known that $\mathcal{M}(a)-1$ is a sum of monomials of degree $\geq q-1$ if and only if $a\in F_{q-1}$ (see Section~5 in~\cite{MKS}). Let $\mathcal{M}(a)=1+\sum_{k\geq s}^{\infty} h_k$ for monomials $h_k$ of degree~$k$.
$$\mathcal{M}\big([x_i^{\vphantom{}},a]\big) = 1+(X_i h_s-h_s X_i) + \sum(\textrm{monomials of degree}> s+1).$$
Since $[x_i^{\vphantom{}},a]\in F_q$, $\deg({X_i h_s-h_s X_i })=s+1\geq q$ and $s\geq q-1$. Thus, $a\in F_{q-1}$.
\end{proof}
On $\H_{g,n}$, $\tilde\mu_q^{\vphantom{}}$ is equivalent to $(\eta_q^{\vphantom{}},\mu_q^{\vphantom{}})$, namely, the invariant $\tilde\mu_q^{\vphantom{}}$ can be thought just as a combination of $\mu_q^{\vphantom{}}$ and~$\eta_q^{\vphantom{}}$.

We consider the image of $\eta_q^{\vphantom{}}$. For $g\geq0$ and $n\geq 1$, let 
{\setlength\arraycolsep{1pt}
\begin{align*}
  \Aut_2(F/F_q):=&\{ \phi \in \Aut(F/F_q)~|~\phi(x_i^{\vphantom{}})=\bar{\mu}_i^{-1}x_i^{\vphantom{}} \bar{\mu}_i^{\vphantom{}} \textrm{ for some }\bar{\mu}_i^{\vphantom{}} \in F/F_{q-1} \\  
  &\textrm{ and }\textrm{there is a lift }F/F_{q+1}\xrightarrow{\tilde\phi} F/F_{q+1} \textrm{ such that }\tilde\phi([\partial_n^{\vphantom{}}])=[\partial_n^{\vphantom{}}]  \} .
\end{align*}

\begin{proposition}
Let $M$ be a homology cylinder. Then $\eta_q^{\vphantom{}}(M) \in \Aut_2(F/F_q)$.
\end{proposition}

\begin{proof}
Since $\eta_{q+1}^{\vphantom{}}(M)$ is a lift of $\eta_q^{\vphantom{}}(M)$ on $F/F_{q+1}$, it follows from $\eta_{q+1}^{\vphantom{}}(M)([\partial_n^{\vphantom{}}])=[\partial_n^{\vphantom{}}]$ and (3).  
\end{proof}

\begin{remark} 
\label{remark:GL}
When $n = 1$ or $g = 0$, the image of $\eta_q^{\vphantom{}}$ is known:
 \begin{enumerate}
    \item   In \cite{GL}, it was shown that $\eta_q^{\vphantom{}} \colon \H_{g,1} \to \Aut_0(F/F_q)$ is surjective where
    \begin{align*}
  \Aut_0(F/F_q):=\{ \phi \in \Aut(F/F_q)~|~&\textrm{there is a lift }F/F_{q+1}\xrightarrow{\tilde\phi} F/F_{q+1} \\ &\textrm{such that }\tilde\phi([\partial_n^{\vphantom{}}])=[\partial_n^{\vphantom{}}]  \} .
\end{align*}
    \item  In \cite{HL98}, it was shown that $\eta_q^{\vphantom{}} \colon \H_{0,n} \to \Aut_1(F/F_q)$ is surjective where
    \begin{align*}
  \Aut_1(F/F_q):=\{ \phi \in \Aut(F/F_q)~|~&\phi(x_i^{\vphantom{}})=\bar{\mu}_i^{-1}x_i^{\vphantom{}} \bar{\mu}_i^{\vphantom{}} 
  \textrm{ for some }\bar{\mu}_i^{\vphantom{}} \in F/F_{q-1} \\
  &\textrm{ and }\phi(x_1^{\vphantom{}}\cdots x_{n-1}^{\vphantom{}})=x_1^{\vphantom{}}\cdots x_{n-1}^{\vphantom{}} \} .
\end{align*}
  \end{enumerate}
\end{remark}
It remains an open question whether $\eta_q^{\vphantom{}} \colon \H_{g,n} \to \Aut_2(F/F_q)$ is surjective.

\subsection{Crossed homomorphisms}
  \label{sec:product formula}
  We remind the reader that $\tilde\mu_q^{\vphantom{}}\colon \H_{g,n}\to (F/F_q)^{2g+n-1}$ is not a homomorphism, although $\eta_q^{\vphantom{}}\colon \H_{g,n}\to \Aut(F/F_q)$ is a homomorphism. However, we have a product formula as follows:

\begin{proposition}
  \label{proposition:product formula}
  Let $M$ and $N$ be homology cylinders over $\Sigma$. Then, $\tilde\mu_q^{\vphantom{}}$ is a crossed homomorphism on $\H_{g,n}$ in the sense that each coordinate $\tilde\mu_q^{\vphantom{}}(-)_k^{\vphantom{}}$ satisfies
  $$\tilde\mu_q^{\vphantom{}}(M\cdot N)_k^{\vphantom{}}=\tilde\mu_q^{\vphantom{}}(M)_k^{\vphantom{}}\cdot \eta_q^{\vphantom{}}(M)(\tilde\mu_q^{\vphantom{}}(N)_k^{\vphantom{}})$$
for $k=1,\ldots, 2g+n-1$.
\end{proposition}

\begin{proof}
   Let $\imath_M\colon M\to M\cdot N $ and $\imath_N \colon N\to M\cdot N$ be the natural inclusions.
   Then $$(M, i_+^{\vphantom{}},i_-^{\vphantom{}})\cdot (N,j_+^{\vphantom{}},j_-^{\vphantom{}})=(M\cup_{i_-\circ (j_+)^{-1}} N,\imath_M\circ  i_+^{\vphantom{}},\imath_N \circ j_-^{\vphantom{}})\quad \textrm{and} $$ 
   $$\tilde\l(M\cdot N)_k^{\vphantom{}} = (\imath_M)_*(\tilde\l(M)_k^{\vphantom{}})\; (\imath_N)_*(\tilde\l(N)_k^{\vphantom{}})\quad \textrm{ in }\pi_1(M\cdot N).$$    
    $$\begin{diagram}
  	\node[2]{M}	\arrow{s,r}{\imath_M} \\
  	\node{\Sigma} \arrow{ne,l}{i_-^{\vphantom{}}} \arrow{se,r}{j_+^{\vphantom{}}}	\node{M\cdot N} \\
  	\node[2]{N}	\arrow{n,r}{\imath_N}
  \end{diagram}$$
 
  The above diagram commutes and all the maps induce isomorphisms on $\pi_1(-)/\pi_1(-)_q$. Thus
  \begin{align*}
  	\tilde\mu_q^{\vphantom{}}(M\cdot N)_k &=  (\imath_M \circ i_+^{\vphantom{}})_{*q}^{-1}(\tilde\l(M\cdot N)_k^{\vphantom{}}) \\
  	&=(i_+^{\vphantom{}})_{*q}^{-1}\circ (\imath_M)_{*q}^{-1}( (\imath_M)_*(\tilde\l(M)_k^{\vphantom{}})\; (\imath_N)_*(\tilde\l(N)_k^{\vphantom{}})) \\
  	&=(i_+^{\vphantom{}})_{*q}^{-1}(\tilde\l(M)_k^{\vphantom{}}) \; \big((i_+^{\vphantom{}})_{*q}^{-1}\circ (\imath_M)_{*q}^{-1}\circ (\imath_N)_{*q}^{\vphantom{}} \big)(\tilde\l(N)_k^{\vphantom{}}) \\
  	&=\tilde\mu_q^{\vphantom{}}(M)_k^{\vphantom{}} \; \big((i_+^{\vphantom{}})_{*q}^{-1}\circ (i_-^{\vphantom{}})_{*q}^{\vphantom{}} \circ (j_+^{\vphantom{}})_{*q}^{-1}\big)(\tilde\l(N)_k^{\vphantom{}}) \\
  	&=\tilde\mu_q^{\vphantom{}}(M)_k^{\vphantom{}}\; \big(\eta_q^{\vphantom{}}(M) \circ (j_+^{\vphantom{}})_{*q}^{-1}\big)(\tilde\l(N)_k^{\vphantom{}}) \\
  	&=\tilde\mu_q^{\vphantom{}}(M)_k^{\vphantom{}}\; \eta_q^{\vphantom{}}(M) (\tilde\mu_q^{\vphantom{}}(N)_k^{\vphantom{}}).	\qedhere
  \end{align*}
\end{proof}

In the remaining part of this paper, as an abuse of notation, we also write $i_\pm^{\vphantom{}} $ and $\tilde\l(M)_k^{\vphantom{}}$ for $\imath_M\circ i_\pm^{\vphantom{}}$ and $ (\imath_M)_*(\tilde\l(M)_k^{\vphantom{}})$, respectively.

From the above proposition, we obtain some properties of $\tilde\mu_q^{\vphantom{}}$:
\begin{corollary}
  \leavevmode \Nopagebreak
  \label{cor:subgroup}
  \begin{enumerate}
  	\item For a homology cylinder $M$, $\tilde\mu_q^{\vphantom{}}(-M)_k^{\vphantom{}}= \eta_q^{\vphantom{}}(M)^{-1}(\tilde\mu_q^{\vphantom{}}(M)_k^{-1})$.
  	\item The kernel of $\tilde\mu_q^{\vphantom{}}(-)_k^{\vphantom{}}$ is a subgroup of $\H$ for each~$k$.
Moreover, $\ker \tilde\mu_q^{\vphantom{}}$ is a normal subgroup of $\H$.
  	\item $\tilde\mu_q^{\vphantom{}}$ is a homomorphism on $\Ker \eta_q^{\vphantom{}}$, and $\tilde\mu$ is a homomorphism on~$\bigcap_q \Ker \eta_q^{\vphantom{}} $.
  	\item  $\tilde\mu_q^{\vphantom{}}$ is a homomorphism on~$\Ker \tilde\mu_{q-1}^{\vphantom{}}$, or more generally, on $\Ker \mu_2^{\vphantom{}}\cap \Ker \eta_{q-1}$. \qedhere
  \end{enumerate}
\end{corollary}

\begin{proof}
		(1) easily follows from $\tilde\mu_q^{\vphantom{}}(M\cdot M^{-1})_k^{\vphantom{}}=1$.
		For (2), let $M$ and $N$ be homology cylinders over $\Sigma$.
		First, if $\tilde\mu_q^{\vphantom{}}(M)_k^{\vphantom{}}=1= \tilde\mu_q^{\vphantom{}}(N)_k^{\vphantom{}}$, then $\tilde\mu_q^{\vphantom{}}(M\cdot N^{-1})_k^{\vphantom{}} = 1$. Next we check that if $\tilde\mu_q^{\vphantom{}}(M)_k^{\vphantom{}}=1$ for all $k$, then $\tilde\mu_q^{\vphantom{}}(N\cdot M\cdot N^{-1})_k^{\vphantom{}}=1$:
		\begin{align*}
			\tilde\mu_q^{\vphantom{}}(N\cdot M\cdot N^{-1})_k^{\vphantom{}}&=\tilde\mu_q^{\vphantom{}}(N)_k^{\vphantom{}} \; \eta_q^{\vphantom{}}(N)\big(\tilde\mu_q^{\vphantom{}}(M)_k^{\vphantom{}} \;\eta_q^{\vphantom{}}(M)(\tilde\mu_q^{\vphantom{}}(N^{-1})_k^{\vphantom{}})\big) \\
			&=\tilde\mu_q^{\vphantom{}}(N)_k^{\vphantom{}} \; \eta_q^{\vphantom{}}(N) \big( \tilde\mu_q^{\vphantom{}}(N^{-1})_k^{\vphantom{}} \big)	\quad \textrm{ due to } \tilde\mu_q^{\vphantom{}}(M)=1 \\
			&=\tilde\mu_q^{\vphantom{}}(N)_k^{\vphantom{}} \;\eta_q^{\vphantom{}}(N) \big( \eta_q^{\vphantom{}}(N)^{-1} (\tilde\mu_q^{\vphantom{}}(N)^{-1}_k \big)	\quad \textrm{ by (1) }	\\
			&= \tilde\mu_q^{\vphantom{}}(N)_k^{\vphantom{}} \; \tilde\mu_q^{\vphantom{}}(N)^{-1}_k	\\
			&=1.
		\end{align*}		
		(3) follows directly from Proposition~\ref{proposition:product formula}. 

For (4), we need the following algebraic fact: for a group $G$, if an automorphism of $G/G_q$ induces the identity on $G/G_{q-1}$, then its restriction on $G_{2}/G_q$ is also the identity. We give a proof: from the hypothesis, for such an automorphism $\phi$ and $g\in G/G_q$, $\phi(g)=ga$ for some $a\in G_{q-1}/G_q$. Since $G_{q-1}/G_q$ is in the center of $G/G_q$, the automorphism restricted on $G_2/G_q$ is the identity. 
From the algebraic fact, we obtain the desired conclusion since $\eta_{q-1}^{\vphantom{}}(M)=\id$ and $\tilde\mu_q^{\vphantom{}}(M)_k^{\vphantom{}}\in F_2/F_q$ for $M\in \Ker \tilde\mu_{q-1}^{\vphantom{}}$ or $\Ker \mu_2^{\vphantom{}}\cap \Ker \eta_{q-1}^{\vphantom{}}$.
\end{proof}

\section{Several subgroups and filtrations of $\H_{g,n}$}
\label{sec:subgroups}

\subsection{Filtrations via extended Milnor invariants}

We introduce a filtration of $\H$
$$\cdots \subset \H(q+1) \subset \H(q)\subset \H(q-1)\subset \cdots \subset \H(2)\subset \H(1)=\H$$
where $\H(q)$ is the normal subgroup $\Ker \tilde\mu_q^{\vphantom{}}$ of $\H$, by Corollary~\ref{cor:subgroup}(2).

Comparing with $\H[q]:= \Ker\eta_q^{\vphantom{}}$, we have $\H(q)\subset \H[q]$. For $n = 1$, the two filtrations are the same.

In \cite{L01}, Levine showed that there is an injective homomorphism from the framed $g$-string link concordance group $\H_{0,g+1}$ to $\H_{g,1}$ and that it induces injections between the successive quotients of the filtration $\{\H_{0,g+1}(q)\}$ and those of the filtration $\{H_{g,1}[q]\}$. 
We generalize this to $\H_{g,n}$ as follows:
\begin{theorem}
  \label{theorem:whole injection}
  Suppose $\imath\colon \Sigma_{g,n} \hookrightarrow \Sigma_{g',n'}$ is an embedding ($n,n'\geq 1$). Then it induces a homomorphism $\tilde\imath\colon \H_{g,n} \to \H_{g',n'}$ and a function $f\colon (F/F_q)^{2g+n-1} \to (F'/F'_q)^{2g'+n'-1}$ which make the following diagram commute:
$$  \begin{diagram}
  	\node{\H_{g,n}} \arrow{e,t}{\tilde\imath} \arrow{s,r}{\tilde\mu_q^{\vphantom{}}} \node{\H_{g',n'}} \arrow{s,l}{\tilde\mu_q^{\vphantom{}}} \\
  	\node{(F/F_q)^{2g+n-1}} \arrow{e,t}{f} \node{(F'/F'_q)^{2g'+n'-1}}
  \end{diagram}$$
  where $F=\pi_1(\Sigma_{g,n})$ and $F'=\pi_1(\Sigma_{g',n'})$.
  Moreover, if each component of $\Sigma_{g',n'}-\Sigma_{g,n}$ has at least one closed boundary, then $f$ is 1-1, and hence $\tilde\mu_q^{\vphantom{}}(\tilde\imath(M))$ determines $\tilde\mu_q^{\vphantom{}}(M)$ for every $M\in\H_{g,n}$.
  If, in addition, at most one component of $\overline{\Sigma_{g',n'}-\Sigma_{g,n}}$ has a disconnected intersection with $\Sigma_{g,n}$, then $\tilde\imath \colon \H_{g,n} \to \H_{g',n'}$ is injective.  
\end{theorem}

\begin{proof}
  We simply write $\Sigma:=\Sigma_{g,n}$, $\Sigma':=\Sigma_{g',n'}$ and $\H:=\H_{g,n}$, $\H':=\H_{g',n'}$. By Proposition~\ref{proposition:effect of change}, we may assume that the generating sets of $F$ and $F'$ are as follows. We can assume that each component of $\partial\Sigma$ maps to either the interior of $\Sigma'$ or the boundary of $\Sigma'$. Let $S_r$ be a component of $\overline{\Sigma'-\Sigma}$, which is a surface of genus $g_r$ with $a_r$ boundaries on $S_r\cap \Sigma$ and $n_r$ boundaries on $S_r-\Sigma$. \\
\textbf{Step 1.} Choose basepoints $\ast$ of $\Sigma$ and $\ast'$ of~$\Sigma'$. \\
	To define $\tilde\mu_q^{\vphantom{}}$, we should choose basepoints on boundaries of the surfaces. There are two cases:
	\begin{itemize}
		\item[Case 1.] Two basepoints can be chosen as the same point, i.e.\ there is a common boundary of $\Sigma$ and $\Sigma'$. Choose $\ast=\ast'$ on the boundary.
		\item[Case 2.] Two basepoints cannot be chosen as the same point, i.e.\ there is no common boundary of $\Sigma$ and $\Sigma'$. Then, at least one $S_r$ has a boundary of $\Sigma'$. Choose $\ast$ and $\ast'$ in the same $S_r$. Choose a path $\gamma$ from $\ast'$ to $\ast$ in the $S_r$ for later use.
	\end{itemize}
\textbf{Step 2.} Choose a generating set for $F$. \\

			We fix a generating set for $F$ as in Section~\ref{sec:representation}. Denote the generators corresponding to the boundaries of $S_r \cap \Sigma$ by $x^r_k$ and the other generators by $y_s^{\vphantom{}}$. \\ 
\textbf{Step 3.} Choose a generating set for $F'$. \\
		First we will choose a generating set of $\pi_1(\Sigma',\ast)$.
		Note that $x^r_k=[\alpha^r_k \cdot\beta^r_k \cdot (\alpha^r_k)^{-1}]$ for a closed path $\beta^r_k$ onto the corresponding boundary and a path $\alpha^r_k$ from $\ast$ to $\beta^r_k(0)$. Let $\imath_*\colon F\to \pi_1(\Sigma',\ast)$ be the map induced by $\imath$. We choose $y_s':=i_*(y_s^{\vphantom{}})$. Now we select the other generators of $\pi_1(\Sigma',\ast)$ with regard to $S_r$. We consider two cases (see Figure~\ref{figure:S_r}):
		\begin{enumerate}
			\item[Case 1.] $S_r \not\ni \ast$ \\
  	Choose generators of $\pi_1(S_r,\alpha^r(1))$ corresponding to the $n_r$ boundary components of $\Sigma'$ in $S_r$ and the $g_r$ handles of $S_r$ as in Section~\ref{sec:representation}, and conjugate them by $\alpha^r_1$ so that we obtain generators $x'^r_i, m'^r_j, l'^r_j$ of $\pi_1(\Sigma',\ast)$ for $i\leq n_r$ and $j\leq g_r$.
  	Let $m'^r_{g_r+k}:=\imath_*(x^r_k)$ and $l'^r_{g_r+k}:=[\alpha^r_k \cdot \gamma^r_k \cdot(\alpha^r_{k+1})
^{-1}]$ for a path $\gamma^r_k$ from $\alpha^r_k(1)$ to $\alpha^r_{k+1}(1)$ in $S_r$ and $k=1,\ldots ,a_r-1$.
			\item[Case 2.] $S_r \ni \ast$ (and $\ast'$) \\
  	Choose generators of $\pi_1(S_r,\ast)$ corresponding to the $n_r-1$ boundary components of $\Sigma'$ in $S_r$ except the one containing $\ast'$, and those corresponding to the $g_r$ handles of $S_r$ as in Section~\ref{sec:representation}. They give generators $x'^r_i, m'^r_j, l'^r_j$ of $\pi_1(\Sigma',\ast)$ for $i\leq n_r-1$ and $j\leq g_r$. Let $m'^r_{g_r+k}:=\imath_*(x^r_k)$ and $l'^r_{g_r+k}:=[\alpha^r_k \cdot (\gamma^r_k)^{-1}]$ for a path $\gamma^r_k$ from $\ast$ to $\alpha^r_k(1)$ in $S_r$ and $k=1,\ldots ,a_r-1$.
		\end{enumerate} 
		We define a set $A_r:=\{x'^r_i, m'^r_j, l'^r_j, m'^r_{g_r+k}, l'^r_{g_r+k}\}$.
		Then $(\bigcup_r A_r) \cup \{y_s'\}$ is a generating set for $\pi_1(\Sigma',\ast)$.
			If $\ast \neq \ast'$, replace all the generators by conjugation by $\gamma$ to obtain a generating set for $F'$. Note that we have $\imath_\#\colon F\xrightarrow{\imath_*} \pi_1(\Sigma',\ast) \xrightarrow{\cong} F'$ where the latter is induced by $\gamma$-conjugation.

\begin{figure}[h]
 \begin{center}
\begin{tabular}{cc}
   \includegraphics[scale=.6]{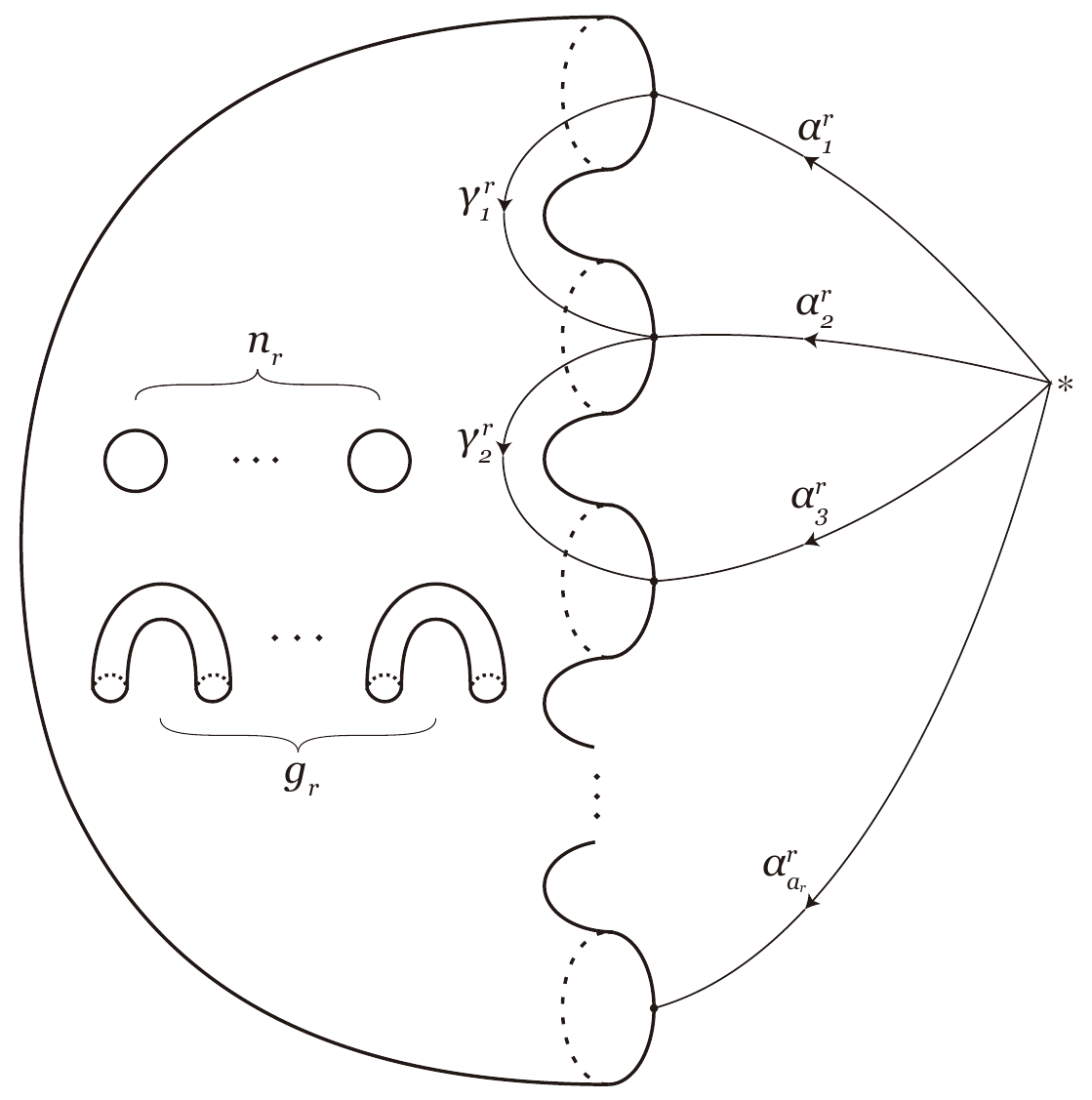} & \includegraphics[scale=.6]{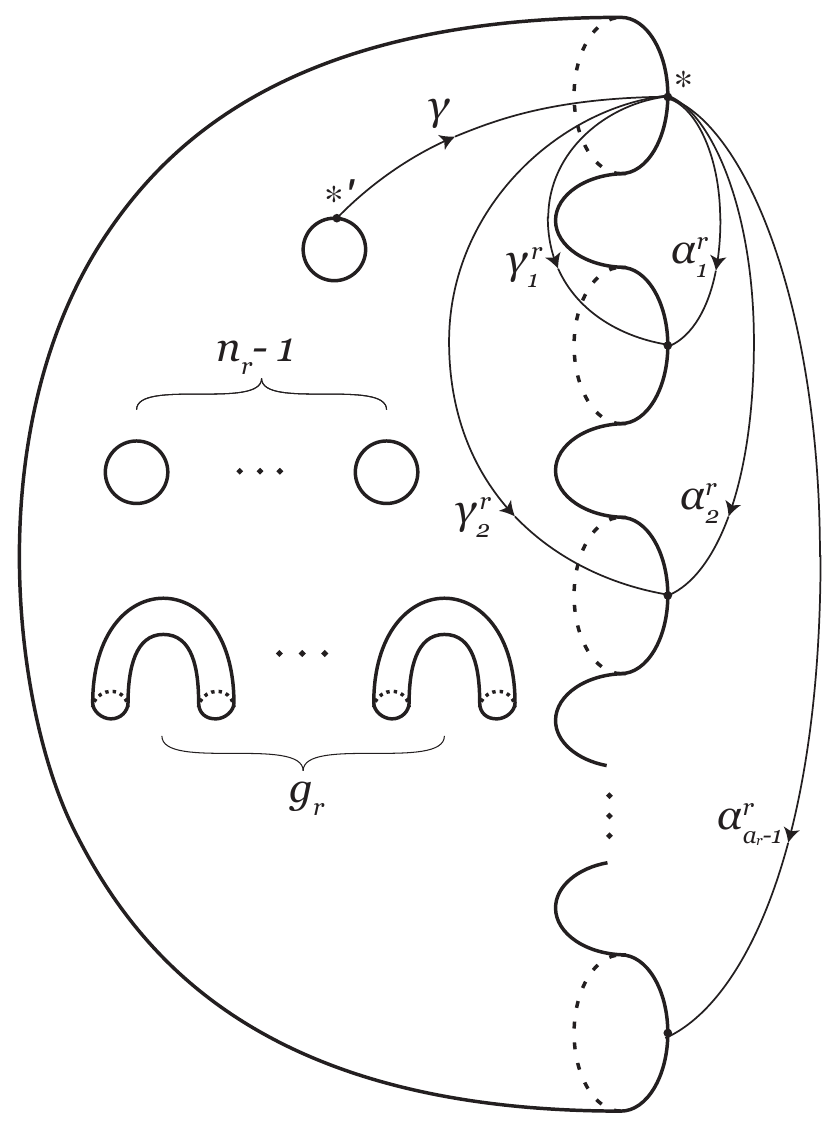} \\
    Case 1.  $S_r \not \ni \ast$  &  Case 2. $S_r \ni \ast$
\end{tabular}
	\caption{$S_r$ and paths to choose generators of $F'$ in Step 3.}
    \label{figure:S_r}
  \end{center}
\end{figure}

Now we define $f$.
We use the following indexing convention for coordinates of elements of $(F/F_q)^{2g+n-1}$, using generators of~$F$.
Recall that the generating set $\{x_k^r, y_s^{\vphantom{}}\}_{r,k,s}$ has been ordered to define $\tilde\mu_q^{\vphantom{}}(-)_k^{\vphantom{}}$ using the $k$th generator. If $a \in \{x_k^r, y_s^{\vphantom{}}\}$ is the $k$th generator, we call the $k$th coordinate of an element in $(F/F_q)^{2g+n-1}$ the \emph{coordinate associated with}~$a$.

The map $f$ can be defined coordinatewise as follows. Let $z\in (F/F_q)^{2g+n-1}$. \\
(1) The coordinates $\eta^r_i,\eta'^r_j,\eta''^r_j$ of $f(z)$ associated with $x'^r_i, m'^r_j, l'^r_j$ are determined by the coordinates $ z_i^{\vphantom{}}$ of $z$ associated with $x_i^r$ for all~$i,j$.
\begin{enumerate}
		\item[Case 1.] $S_r \not\ni \ast$ 
		\begin{eqnarray*}
			\eta^r_i=\imath_\#(z_1^{\vphantom{}}) , \quad
			\eta'^r_j=[m'^r_j,\imath_\#(z_1^{\vphantom{}})^{-1}],\quad 
			\eta'^r_{g_r+k}=[m'^r_{g_r+k},\imath_\#(z_k^{\vphantom{}})^{-1}] , \\
			\eta''^r_j=[l'^r_j,\imath_\#(z_1^{\vphantom{}})^{-1}], \quad
			\eta''^r_{g_r+k}=l'^r_{g_r+k}\;\imath_\#(z_{k+1}^{\vphantom{}})^{-1} \;(l'^r_{g_r+k})^{-1} \; \imath_\#(z_k^{\vphantom{}}) 
		\end{eqnarray*}
		for $i=1,\ldots, n_r$, $j=1,\ldots, g_r$, and $k=1,\ldots, a_r-1$.
			 \item[Case 2.] $S_r \ni \ast$ \\
		 If $a_r>1$, then $\eta^r_i, \eta'^r_j, \eta'^r_{g_r+k}, \eta''^r_j$ are the same as Case 1, and 		$\eta''^r_{g_r+k}=\imath_\#(z_k)$
		for $i=1,\ldots, n_r-1$, $j=1,\ldots, g_r$, and $k=1,\ldots, a_r-1$.
		If $a_r=1$, then all the coordinates associated with $x'^r_i, m'^r_j, l'^r_j$ are~$1$ for $i=1,\ldots, n_r-1$, $j=1,\ldots, g_r$.
	\end{enumerate}
(2) The coordinate $\eta_s^{\vphantom{}}$ of $f(z)$ associated with $y_s'$ is determined by the coordinate $z_s^{\vphantom{}}$ of $z$ associated with $y_s^{\vphantom{}}$.
$$\eta_s^{\vphantom{}}=\imath_\#(z_s^{\vphantom{}}) \textrm{ for all } s.$$

Remark that $f$ is not a homomorphism. 
From the definition of $f$, it is 1-1 if $n_r \geq 1$ for all $r$ and $\imath_\#$ is injective. If every $n_r$ is nonzero, then $\imath_\#$ is injective. Therefore, if every $n_r$ is positive, then $f$ is 1-1.
The verification that the diagram commutes is left to the reader.

Suppose at most one $a_r$ is bigger than 1.
To prove $\tilde\imath\colon \H\to\H'$ is injective, we claim that there is a function $\tilde\jmath \colon \H'\to\H$ so that $\tilde\jmath \circ \tilde\imath = \id_\H$. The surface $\Sigma'$ is obtained from $\Sigma$ by attanching 1-handles to a collar neighborhood of $S_r\cap \Sigma$. This allows us to extend $\Sigma \xrightarrow{\id\times \frac{1}{2}} \Sigma\times I$ to an embedding $\jmath\colon \Sigma' \hookrightarrow \Sigma\times I$.
For example, see Figure~\ref{figure:embedding}.
\begin{figure}[h]
  \begin{center}
    \includegraphics[scale=.55]{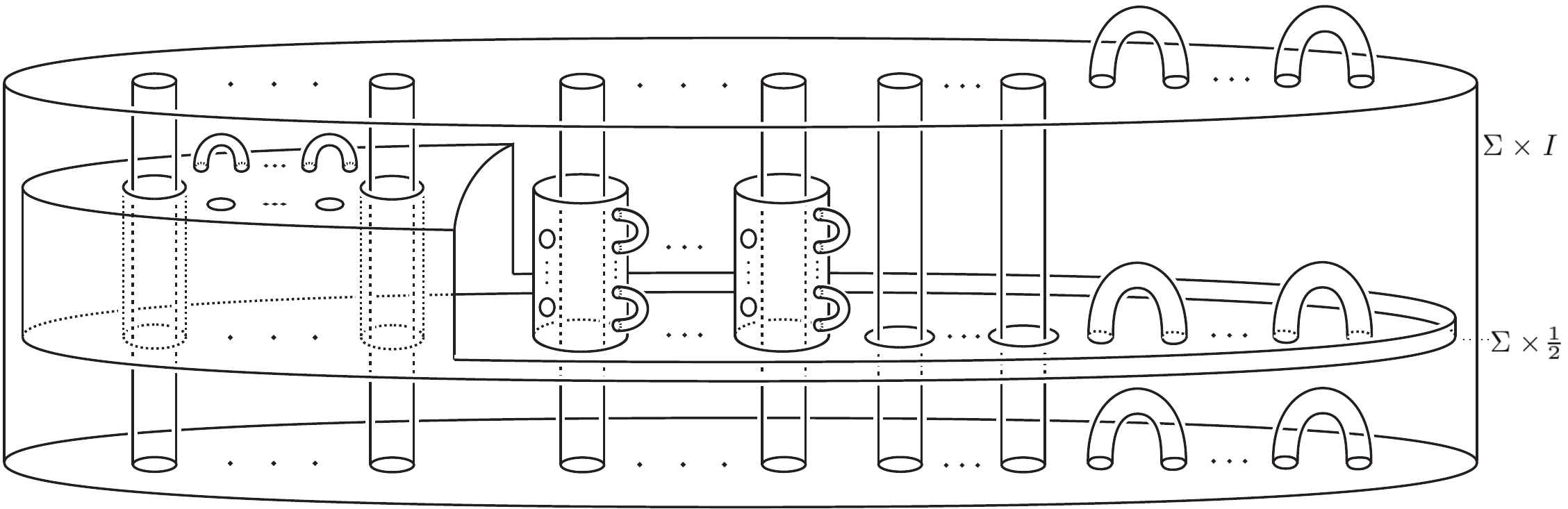}
    \caption{$\jmath (\Sigma') $ in $\Sigma\times I$}
    \label{figure:embedding}
  \end{center}
\end{figure}
We define $\tilde\jmath (M')$ as the manifold obtained by cutting $\Sigma\times I$ open along $\jmath(\Sigma')$ and filling in it with $M'$.
It is easy to check that $\tilde\jmath\colon \H' \to \H$ is well-defined. We obtain the injectivity of $\tilde\imath$ from $\tilde\jmath \circ \tilde\imath = \id_\H$.
\end{proof}

Whenever $f$ is 1-1, $\tilde\imath$ induces injections $\H(q-1)/\H(q) \hookrightarrow \H'(q-1)/\H'(q)$. By considering the cases of $(g,n)$ and $(g',n')$ for which there is an injection $\H_{g,n}\to\H_{g',n'}$, we obtain the following corollary. Note that if $\Sigma_{g,n}\subset \Sigma_{g',n'}$, then $g\leq g'$.

\begin{corollary}
  \label{cor:whole injection}
  For any two pairs $(g,n)$ and $(g',n')$ satisfying $g \leq g'$ and $g+n \leq g'+n'$, there is an injective homomorphism
  $$\H_{g,n} \hooklongrightarrow \H_{g',n'} ,$$
which induces injections
  $$\H_{g,n}(q-1)/\H_{g,n}(q) \hooklongrightarrow \H_{g',n'}(q-1)/\H_{g',n'}(q)$$
for all $q \geq 2$.
\end{corollary}
\begin{proof}
There exists an embedding $\imath\colon \Sigma_{g,n}\hookrightarrow\Sigma_{g',n'}$ such that both $f$ and $\tilde\imath$ are injective: $\overline{\Sigma_{g',n'}-\Sigma_{g,n}} =S_1$ is connected, and if  $n>n'$, $a_1=n-n'+1$, $n_1=1$, and $g_r=g'+n'-g-n$; otherwise, $a_1=1$, $n_1=n'-n+1$, and $g_r=g'-g$. 
The conclusion follows from Theorem~\ref{theorem:whole injection}
\end{proof}

 The following theorem gives more information on the group $\H(q-1)/\H(q)$ directly via~$\tilde\mu_{q}^{\vphantom{}}$: 
\begin{theorem}
  \leavevmode \Nopagebreak \label{theorem:rank}
	\begin{enumerate}
	  \item $\tilde\mu_{q}^{\vphantom{}}$ induces an injective homomorphism
	  $$\tilde\mu_{q}^{\vphantom{}}\colon \H(q-1)/\H(q)\hooklongrightarrow  (F_{q-1}/F_{q})^{2g+n-1}.$$
	  Hence $\H(q-1)/\H(q)$ is a finitely generated free abelian group.
	\item We have
			$$\max\{r_q(2g),r_q(g+n-1)\} \leq \rank \H(q-1)/\H(q) \leq r_q(2g+n-1)$$
		where $N_q(m)=\frac{1}{q}\sum_{d|q} \varphi(d) (m^{q/d})$, $\varphi$ is the M\"obius function
		and $r_q(m)=m N_{q-1}(m)-N_q(m)$.
	\end{enumerate}
\end{theorem}

Note that 
\begin{align*}
N_q(2g+n-1) &= \rank F_q/F_{q+1} \qquad \textrm{and} \\
r_q(2g+n-1)&=(2g+n-1)\;\rank F_{q-1}/F_q -\rank F_q/F_{q+1} \\
&=\Coker\{H_3(F/F_{q})\to H_3(F/F_{q-1})\}.
\end{align*} 

We remark that the facts 
\begin{align*}
\rank \H_{g,1}(q-1)/\H_{g,1}(q) &= r_q(2g)\qquad \textrm{and} \\
\rank \H_{0,n}(q-1)/\H_{0,n}(q) &= r_q(n-1)
\end{align*}
 were shown in \cite{GL} and \cite{Or}, respectively.

\begin{proof}
\begin{enumerate}
  	\item  Since $\tilde\mu_{q}^{\vphantom{}}\colon \H(q-1)\to (F_{q-1}/F_{q})^{2g+n-1}$ is a homomorphism, for $M, N \in \H(q-1)$, $[M]=[N]$ in $\H(q-1)/\H(q)$ if and only if $\tilde\mu_{q}^{\vphantom{}}(M)= \tilde\mu_{q}^{\vphantom{}}(N)$. The well-definedness and the injectivity of the map follow.
	\item Consider the map
\begin{align*}
\mathfrak{p}\colon (F_{q-1}/F_{q})^{2g+n-1}&\to F_{q}/F_{q+1} \\
(a_1^{\vphantom{}},\ldots,a_{n-1}^{\vphantom{}},b_1^{\vphantom{}},\ldots, b_g^{\vphantom{}}, c_1^{\vphantom{}},\ldots, c_g^{\vphantom{}}) &\longmapsto\prod_{i=1}^{n-1} [x_i^{\vphantom{}},a_i^{\vphantom{}}]\prod_{j=1}^{g} [l_j^{\vphantom{}},b_j^{\vphantom{}}][b_j^{\vphantom{}},c_j^{\vphantom{}}][c_j^{\vphantom{}},m_j^{\vphantom{}}].
\end{align*}
 This $\mathfrak{p}$ is a surjective homomorphism, and the kernel has rank $r_q(2g+n-1)$.  For the upper bound of the rank of $\H(q-1)/\H(q)$, we claim that $\mathfrak{p}\circ \tilde\mu_{q}^{\vphantom{}}$ is trivial on~$\H(q-1)$. For any homology cylinder $M$, $\eta_{q}^{\vphantom{}}(M)$ fixes $[\partial_n^{\vphantom{}}]$ for all~$q$. Using the generating set of $F$ in Figure~\ref{figure:Sigma}, the element $[\partial_n^{\vphantom{}}] = \prod_i x_i^{\vphantom{}}\prod_j [m_j^{\vphantom{}},l_j^{\vphantom{}}] \in F$. For $M \in \H(q-1)$, $\eta_{q+1}^{\vphantom{}}(M)([\partial_n^{\vphantom{}}])=[\partial_n^{\vphantom{}}]$ is arranged to 
$$\prod_i[x_i^{\vphantom{}},\mu_{q}^{\vphantom{}}(M)_i^{\vphantom{}}]\prod_j[l_j^{\vphantom{}},\mu_{q}'(M)_j^{\vphantom{}}][\mu_{q}'(M)_j^{\vphantom{}},\mu_{q}''(M)_j^{\vphantom{}}][\mu_{q}''(M)_j^{\vphantom{}},m_j^{\vphantom{}}]=1.$$ 
Therefore $\mathfrak{p}\circ \tilde\mu_{q}^{\vphantom{}}=\mathfrak{p}\circ (\mu_{q}^{\vphantom{}},\mu_{q}',\mu_{q}'')$ is trivial. 	
The lower bound comes from Corollary~\ref{cor:whole injection} and the known ranks of $\H_{g,1}(q-1)/\H_{g,1}(q)$ and $\H_{0,n}(q-1)/\H_{0,n}(q)$ \cite{GL,Or} stated above.  \qedhere 
	\end{enumerate}
\end{proof}

We define $\H^0[q]:=\{M\in\H[q]~|~\mu_2(M)=1\}$. Then it is a subgroup of $\H$ by Corollary~\ref{cor:subgroup}, and $\H(q) \subset \H^0[q]\subset \H(q-1)$ by Lemma~\ref{lemma:MKS}.
Thus we can refine the filtration~$\{\H(q)\}$ of $\H$ as follows:
$$\cdots \subset \H(q) \subset \H^0[q] \subset \H(q-1) \subset \H^0[q-1] \subset \cdots \subset \H(2) \subset \H^0[2] \subset \H(1)=\H$$
Consider two injections $\mu_q^{\vphantom{}}$ and $(\mu_q',\mu_q'')$ from the injection $\tilde\mu_{q}^{\vphantom{}}$ as follows:
$$\mu_{q}^{\vphantom{}}\colon \frac{\H^0[q]}{\H(q)} \hooklongrightarrow  (F_{q-1}/F_{q})^{n-1}$$
$$(\mu_{q}',\mu_{q}'') \colon \frac{\H(q-1)}{\H^0[q]} \hooklongrightarrow  (F_{q-1}/F_{q})^{2g}$$
The two subquotients of $\H$ are also finitely generated free abelian, and there is an isomorphism
$$ \frac{\H(q-1)}{\H(q)} \cong \frac{\H(q-1)}{\H^0[q]}\times \frac{\H^0[q]}{\H(q)} ,$$
which is not canonical.

In the remaining part of this section, we introduce some notions analogous to the boundary (string) links and the $\hF$-(string) links.

\subsection{Boundary homology cylinders}
  \label{sec:boundary homology cylinder}
As a generalization of boundary (string) links, we define boundary homology cylinders. 
\begin{definition}
  A homology cylinder $(M,i_+^{\vphantom{}},i_-^{\vphantom{}})$ over $\Sigma$ is said to be a \emph{boundary homology cylinder} if there exists a homomorphism $\phi \colon \pi_1(M) \to \pi_1(\Sigma)$ such that $\phi \circ (i_+^{\vphantom{}})_* = \id = \phi \circ (i_-^{\vphantom{}})_*$ and $\phi(\l_i^{\vphantom{}})=1$ for all~$i$.
\end{definition}

Geometrically, the boundary homology cylinder $M$ can be defined to be a homology cylinder such that there exists $\psi\colon M \to \Sigma$ making the left diagram below commute.
$$\begin{diagram}
    \node{\partial M} \arrow{s,J} \arrow{e,t}{i_+^{-1} \cup i_-^{-1}} \node{\Sigma}\\
    \node{M} \arrow{ne,b,..}{\psi}
\end{diagram}
\qquad \qquad
\begin{diagram}
    \node{i_+^{\vphantom{}}(\Sigma)} \arrow{s,J} \arrow{e,t}{i_+^{-1}} \node{\Sigma}\\
    \node{M} \arrow{ne,b,..}{\psi_+^{\vphantom{}}}
\end{diagram}
$$

Let $(M,i_+^{\vphantom{}},i_-^{\vphantom{}})$ be a homology cylinder.
We consider splittings $\phi_+^{\vphantom{}}$ and $\phi_-^{\vphantom{}}$ of $(i_+^{\vphantom{}})_*$ and $(i_-^{\vphantom{}})_*$, respectively. Note that $\phi_+^{\vphantom{}}$ exists if and only if $\psi_+^{\vphantom{}}$ exists in the above right diagram. The existence of $\phi_+^{\vphantom{}}$ does not imply that of $\phi_-^{\vphantom{}}$. Even if both $\phi_+^{\vphantom{}}$ and $\phi_-^{\vphantom{}}$ exist, they can be different. 
Suppose $\phi_+^{\vphantom{}}$ exists. Then $\eta_q^{\vphantom{}}(M)$ is trivial if and only if $\phi_+^{\vphantom{}}$ is also a splitting of $(i_-^{\vphantom{}})_*^{\vphantom{}}$ due to $\bigcap F_q= 0$. However,  such a common splitting of $(i_+^{\vphantom{}})_*$ and $(i_-^{\vphantom{}})_*$ does not guarantee that  the homology cylinder is a boundary homology cylinder. 
Such an example can be found by considering homology cylinders of the form $(\Sigma\times I/\sim, \id\times 0, \id\times 1 \circ \varphi)$ with nonvanishing $\mu_q^{\vphantom{}}$, where $\varphi$ is a composition of Dehn twists about boundaries. The condition $\phi(\l_i^{\vphantom{}})=1$ for all $i$, or equivalently all $\mu_q^{\vphantom{}}$ vanish, is necessary to satisfy the geometric definition of the boundary homology cylinder. 
In conclusion, we have the following:

\begin{proposition}
  \label{proposition:splitting}
A homology cylinder $M$ is a boundary homology cylinder if and only if the following hold:
\begin{enumerate}
    \item   There is a splitting of $(i_+^{\vphantom{}})_*$ or $(i_-^{\vphantom{}})_*$.
    \item   $\tilde\mu(M)$ vanishes.
\end{enumerate}
\end{proposition}

The subset of boundary homology cylinders are closed under the multiplication and inverting (=orientation reversing and changing two markings) of $\C_{g,n}$, but a homology cylinder which is homology cobordant to a boundary homology cylinder may not be a boundary homology cylinder. (For example, \cite{Sm} provides such a string link.) We define $\B\H$ to be the subgroup of homology cobordism classes of boundary homology cylinders.

Remark that for a framed string link $\sigma$ in a homology 3-ball, its exterior $E_{\sigma}$ is a boundary homology cylinder if and only if its closure $\hat\sigma$ is a boundary link and the framing of $\sigma$ induces the 0-framing of $\hat\sigma$.
It follows from that $\pi_1(E_{\hat\sigma})=\pi_1(E_\sigma)/\langle\langle i_+^{\vphantom{}}(x_i^{\vphantom{}})= i_-^{\vphantom{}}(x_i^{\vphantom{}}) \rangle\rangle$.

\subsection{$\hF$-homology cylinders}

  \label{sec:hF-homology cylinder}
We define $\hF$-homology cylinders as an analog of the $\hF$-(string) links. Here $\hG$ means the algebraic closure of a group $G$ with respect to $\Z$-coefficient or $\zp$-coefficient in the sense of~\cite{C08}. The former was called the $HE$-closure in \cite{L90}. For both $\Z$ and $\zp$, everything in this paper holds.

It is known that for CW-complexes $X$ and $Y$ with finite 2-skeletons, if $X\to Y$ is 2-connected on the integral homology, then it induces an isomorphism on $\widehat{\pi_1(-)}$ \cite{L89, C08}.
Hence the markings $\Sigma\to M$ for a homology cylinder $M$ induce isomorphisms $\hF \xrightarrow{\cong} \widehat{\pi_1(M)}$.

\begin{definition}
A homology cylinder $M$ is called an \emph{$\hF$-homology cylinder} if $\tilde\l_k^{\vphantom{}} \in \pi_1(M)$ vanishes in $\widehat{\pi_1(M)}$ for every $k=1,\ldots,2g+n-1$.
\end{definition}

Note that $M$ is an $\hF$-homology cylinder if and only if $\hat{i}\colon \Sigma \to \hM$ induces an isomorphism on~$\widehat{\pi_1(-)}$. This can be shown using properties of the algebraic closure functor (see the proof of \cite[Proposition~6.3]{C10}).
The $\hF$-homology cylinders form a subgroup of $\H_{g,n}$, by the following lemma:

\begin{lemma}
  \leavevmode \Nopagebreak 
  \label{lemma:hF=closed}
  \begin{enumerate}
    \item    The product of two $\hF$-homology cylinders is an $\hF$-homology cylinder. 
    \item   If $M$ is an $\hF$-homology cylinder, then $(-M)$ is also an $\hF$-homology cylinder.
    \item   A homology cylinder which is homology cobordant to an $\hF$-homology cylinder is an $\hF$-homology cylinder.
  \end{enumerate}
\end{lemma}

\begin{proof}
  \begin{enumerate}
  \item
    Let $M=(M,i_+^{\vphantom{}},i_-^{\vphantom{}})$ and $M'=(M',i_+',i_-')$ be $\hF$-homology cylinders. By the Seifert-van Kampen theorem,
$\pi_1(M\cdot M')$ is the pushout of $(i_+^{\vphantom{}})_*$ and~$(i_-^{\vphantom{}})_*$, and $\tilde\l(M\cdot M')_k^{\vphantom{}}$ is the product of $\tilde\l(M)_k^{\vphantom{}}$ and $\tilde\l(M')_k^{\vphantom{}}$ in $\pi_1(M\cdot M')$. All $\tilde\l(M\cdot M')_k^{\vphantom{}}$ vanish in $\widehat{\pi_1(M\cdot M'})$ by the following commutative diagram:
      $$\begin{diagram}
    \dgARROWLENGTH=.1\dgARROWLENGTH    \dgHORIZPAD=.3em
        \node{\pi_1(\Sigma)} \arrow[2]{s,l}{(i_+')_*} \arrow{se} \arrow[2]{e,t}{(i_-^{\vphantom{}})_*} \node[2]{\pi_1(M)} \arrow{s,-} \arrow{se} \\
    \node[2]{\widehat{\pi_1(\Sigma)}} \arrow[2]{s} \arrow[2]{e} \node{} \arrow{s} \node{\widehat{\pi_1(M)}} \arrow[2]{s}    \\
    \node{\pi_1(M')}  \arrow{se} \arrow{e,-}   \node{} \arrow{e}   \node{\pi_1(M\cdot M')} \arrow{se}  \\
    \node[2]{\widehat{\pi_1(M')}} \arrow[2]{e} \node[2]{\widehat{\pi_1(M\cdot M'})}
  \end{diagram}$$
  \item It is obvious since $\tilde\l(-M)_k^{\vphantom{}}=\tilde\l(M)^{-1}_k$ in $\pi_1(M)$.
  \item 
  Let $(M,i_+^{\vphantom{}},i_-^{\vphantom{}})$ and $(M',i_+',i_-')$ be homology cylinders.
  Suppose $M$ is an $\hF$-homology cylinder and $W$ is a homology cobordism between $M$ and $M'$. Consider the commutative diagram below.
  $$  \begin{diagram} \dgARROWLENGTH=0.6\dgARROWLENGTH  \dgHORIZPAD=3.7ex \dgVERTPAD=2.7ex
    \node{\Sigma \cup_\partial \Sigma} \arrow{s,l,J}{i_+'\cup i_-'} \arrow{e,t,J}{i_+^{\vphantom{}}\cup i_-^{\vphantom{}}} \node{M} \arrow{s,J} \\
    \node{M'} \arrow{e,J} \node{W}
  \end{diagram}  $$
  We should check that the elements $\tilde\l(M')_k^{\vphantom{}}$ of $\pi_1(M')$ vanish in $\widehat{\pi_1(M)}$. Both $\tilde\l(M)_k^{\vphantom{}}$ and $\tilde\l(M')_k^{\vphantom{}}$ come from the same element of $\pi_1(\Sigma\cup_\partial \Sigma)$ along $(i_+^{\vphantom{}}\cup i_-^{\vphantom{}})_*$ and  $(i_+'\cup i_-')_*$ so they are mapped to the same element along isomorphisms $\widehat{\pi_1(M)} \xrightarrow{\cong} \widehat{\pi_1(W)} \xleftarrow{\cong} \widehat{\pi_1(M')}$ induced by inclusions. Since $\tilde\l(M)_k^{\vphantom{}}$ vanishes in $\pi_1(\hM)$ for all $k$, so does~$\tilde\l(M')_k^{\vphantom{}}$.    \qedhere
  \end{enumerate}
\end{proof}

We denote the subgroup of $\hF$-homology cylinders by~$\hH$. 

Levine defined a notion of an $\hF$-link using his algebraic closure which involves a certain normal generation condition in \cite{L89}. We denote Levine's algebraic closure by $\hG^\mathrm{Lev}$ to avoid confusion. The definition is as follows:
a link $L$ is called an \emph{$\hF$-link} if a meridian map into link exterior $E_L$ induces an isomorphism on $\widehat{\pi_1(-)}^\mathrm{Lev}$ and the preferred longitudes vanish in $\widehat{\pi_1(E_L)}^\mathrm{Lev}$. 
In this paper, we use a modified definition by replacing $\hG^\mathrm{Lev}$ by our $\hG$ as in \cite{C10}. 
Levine's $\hF$-link is an $\hF$-link in our sense, though the converse is open.

For a framed string link $\sigma$ in a homology 3-ball, its exterior $E_\sigma$ is an $\hF$-homology cylinder if and only if its closure $\hat\sigma$ is an $\hF$-link (in our sense) and the framing of $\sigma$ induces the 0-framing of $\hat\sigma$. It follows by considering $$\pi_1(E_\sigma)\twoheadrightarrow \pi_1(E_{\hat\sigma})\twoheadrightarrow \pi_1(\widehat{E_{\sigma}})=\pi_1(E_\sigma)/\langle\langle \l_i^{\vphantom{}}\rangle\rangle$$ where $E_{\hat\sigma}$ is the exterior of the link~$\hat\sigma$.  

\begin{remark}
  \leavevmode \Nopagebreak 
  \begin{enumerate}
 	\item Since a boundary homology cylinder is an $\hF$-homology cylinder, $\B\H \subset \hH$. In general, the inclusion is strict because it is known that there are $\hF$-string links which are not concordant to any boundary string link \cite{CO}.
	\item $\hH \subseteq \H(\infty)$, but we do not know the converse. This question is a homology cylinder version of a long-standing conjecture that a link with vanishing $\bar\mu$-invariants is an $\hF$-link in the sense of Levine.
	\item For the mapping class group $\M$ over $\Sigma$, we note that $\M \cap \H(\infty) = 0$. This may be used to find links which are not fibered but homologically fibered as in~\cite{GS}.
  \end{enumerate}
\end{remark}

\section{Hirzebruch-type invariants}
  \label{sec:hirzebruch}

We review some definitions in \cite{C10,C09}. Let $X$ be a CW complex. 
A tower
$$X_{(h)} \to \cdots \to X_{(1)} \to X_{(0)} = X$$
of covering maps is called a \emph{$p$-tower of height $h$} for $X$ if
each $X_{(t+1)} \to X_{(t)}$ is a connected cover whose covering transformation group is a finite
abelian $p$-group.
For a power $d$ of $p$, a \emph{$\Z_d$-valued $p$-structure of height $h$} for $X$ is a pair $(\{X_{(t)}\},\phi)$ of a $p$-tower of height $h$ for $X$ and a surjective character $\phi \colon \pi_1(X_{(h)}) \twoheadrightarrow \Z_d$.
We omit the word ``$\Z_d$-valued'' from now on.
A $p$-structure of height $h$ is equivalent to a $p$-tower of height $h+1$ such that the $(h+1)$st cover is a $\Z_d$-cover of the $h$th cover.
For a $p$-structure of height $h$ for $X$, we usually denote by $X_{(h+1)}$ the $\Z_d$-cover of the top cover $X_{(h)}$ determined by~$\phi$.

We recall the definition of the Hirzebruch-type intersection form defect invariant \cite[Definition~2.2]{C10}.
Let $\T$ be a $p$-structure $(\{M_{(t)}\},\phi)$ of height $h$ for a closed 3-manifold~$M$. Suppose $r(M_{(h)},\phi)=0$ in the bordism group $\Omega_3(B\Z_d)$ for some~$r>0$. Then there is a 4-manifold $W$ bounded by $rM$ over~$\phi$. Choose a subring $\mathcal{R}$ of $\Q$ containing $1/r$. Define
$$\l(M,\T) = \frac{1}{r} \otimes ([\l_{\Q(\zeta_d)}(W)]-[\l_\Q(W)]) \in \mathcal{R} \otimes_\Z L^0(\Q(\zeta_d))$$
where $[\l_\K(W)]$ is the witt class of the $\K$-coefficient intersection form of $W$ for a field~$\K$.

Note that since $\Omega_3(B\Z_d)=H_3(\Z_d)=\Z_d$, some multiple of a closed 3-manifold over $\Z_d$ is zero in the bordism group.

\begin{lemma}[Lemma~4.4 in \cite{C10}]
  \label{lemma:p-torsion-free}
  For a $p$-structure $(\{M_{(t)}\},\phi)$ of height $h$ for a closed 3-manifold $M$,
 if $H_1(M_{(h)})$ is $p$-torsion free, then $(M_{(h)},\phi)$ is null-bordant over $\Z_d$ so that $\l(M,\T)$ is well-defined as an element in $L^0(\Q(\zeta_d))$.
\end{lemma} 
\begin{proof}
  The character $\phi\colon \pi_1(M_{(h)})\to \Z_d$ factors through $\Z$ if $H_1(M_{(h)})$ is $p$-torsion free.
   Since $(M_{(h)},\phi) \in \Im\{\Omega_3(B\Z)\to \Omega_3(B\Z_d)\}$ and $\Omega_3(B\Z)\cong H_3(\Z)=0$, $(\hM_{(h)},\phi)=0$ in $\Omega_3(B\Z_d)$.
\end{proof}

For a homology cylinder $M$, recall that there is an associated closed 3-manifold $\hM$, the closure of $M$. For each $p$-structure $\T$ for $\hM$, $\l(\hM,\T)$ is defined.
We want to define invariants parametrized by the $p$-structures for the base surface $\Sigma$ rather than those for $\hM$, and hope that the invariants are homomorphisms of (subgroups of) the homology cobordism group $\H$.
 To do this, we first investigate how to determine a $p$-structure for $\hM$ from a given $p$-structure for $\Sigma$.
 
\subsection{$p$-structures and $p$-tower maps}
  \label{sec:defining condition}
  
  A map $f \colon X \to Y$ is called a \emph{$p$-tower map} if it gives rise to a 1-1 correspondence 
$$\Phi_f \colon \{p\text{-structures for } Y\} \to \{p\text{-structures for } X\}$$
via pullback. 
For a more precise description, we recall the definition of the pullback cover.
For a covering map $p\colon \tilde{Y}\to Y$ and a map $f\colon X\to Y$, the space
$\tilde{X}=\{(x,a)\in X\times \tilde{Y}~|~f(x)=p(a)\}$ is called the \emph{pullback cover} of $X$ by $f$. The canonical projection map $\tilde{X}\to X$ is a covering map. We note that the fiber of the pullback cover is homeomorphic to that of the original cover.
For a map $X\to Y$ and a $p$-tower for $Y$, if we take the pullback covers inductively, then we get a tower of $p$-covers of $X$, but some covers may be disconnected. Hence $\Phi_f$ is well-defined only when all the pullback covers are connected and the induced character is surjective.

It is known that if $X$ and $Y$ have finite 2-skeletons, and if $f$ is 2-connected or, more generally, induces an isomorphism on $\widehat{\pi_1(-)}$, then $f$ is a $p$-tower map \cite[Lemma~3.7, Proposition~3.9]{C10}. Hence, for example, each marking $\Sigma\to M$ of any homology cylinder $M$ is a $p$-tower map. 

\begin{definition}
  \label{definition:of order}
  Let $X$ and $Y$ be CW-complexes.
  \begin{enumerate}
  	\item A $p$-structure of height $h$ for $X$ is called a \emph{$p$-structure of order $q$} if $\pi_1(X)_q \subset \pi_1(X_{(h+1)})$, viewing $\pi_1(X_{(h+1)})$ as a subgroup of $\pi_1(X)$, via the injection induced by the covering projection. 
  	\item A map $X\to Y$ is called a \emph{$p$-tower map of order $q$} if it induces a 1-1 correspondence 
  	$$\Phi_f \colon \{p\text{-structures for } Y \textrm{ of order }q\} \to \{p\text{-structures for } X \textrm{ of order }q\}$$ via pullback.
  \end{enumerate}
\end{definition}
Note that the defining condition in Definition~\ref{definition:of order} (1) is independent of the basepoints of $X$ and $X_{(h+1)}$.

We need an algebraic lemma.
\begin{lemma}
  \label{lemma:algebra}
  For a finitely generated group $G$, suppose there is a subnormal series $G_{(t)}  \triangleleft  \cdots  \triangleleft  G_{(1)}  \triangleleft  G_{(0)} = G$ whose factor groups are abelian $p$-groups. Then $G_{(t)}$ contains $G_q$ for some~$q$.
\end{lemma}
Note that if $G_{(t)}$ were a normal subgroup of $G$, then the conclusion would be immediate from that any $p$-group is nilpotent. We give a proof for the above general case at the end of this section.

Let $X$ and $Y$ be CW-complexes with finitely generated fundamental groups.
By the lemma,
 $$\{ p \textrm{-structures for } X\} = \bigcup_q\{p \textrm{-structures for } X\textrm{ of order }q\} , $$
and $X\to Y$ is a $p$-tower map if and only if it is a $p$-tower map of order $q$ for all $q$.

\begin{lemma}
  \label{lemma:well-defined and 1-1}
  Let $X$ and $Y$ be connected CW-complexes whose fundamental groups are finitely generated. If $f\colon X\to Y$ induces a surjection on $\pi_1(-)/\pi_1(-)_q$, then $\Phi_f$ between $p$-structures of order $q$ is well-defined and 1-1.
\end{lemma}

\begin{proof}
For the well-definedness, we should check that pullback covers are connected $p$-covers.
Use induction on $t\geq 0$. 
Suppose there exists a unique $Y_{(t)}$ such that the pullback cover is $X_{(t)}$, and $X_{(t)}$ is connected.
Let $Y_{(t+1)}$ be determined by $\pi_1(Y_{(t)})\twoheadrightarrow \Gamma_{(t)}$. Since the homomorphism factors through $\pi_1(Y_{(t)})/\pi_1(Y)_q$, the composition map $\pi_1(X_{(t)})\to \Gamma_{(t)}$ also factors through $\pi_1(X_{(t)})/\pi_1(X)_q \twoheadrightarrow \pi_1(Y_{(t)})/\pi_1(Y)_q$ and is surjective. Hence the pullback cover $X_{(t+1)}$ is a connected cover of $X_{(t)}$ whose cover transformation group is $\Gamma_{(t)}$. 
The map $\pi_1(X_{(t)})\twoheadrightarrow \Gamma_{(t)}$ which determines $X_{(t)}$ factors through $\pi_1(X_{(t)})/\pi_1(X)_q \twoheadrightarrow \pi_1(Y_{(t)})/\pi_1(Y)_q$. Hence it determines $\pi_1(Y_{(t)}) \to \Gamma_{(t)}$ uniquely. Therefore $Y_{(t+1)}$ whose pullback is $X_{(t+1)}$ is unique.
\end{proof}

Applying this lemma to the marking $\hat{i} \colon \Sigma \to \hM$ of a homology cylinder $M$, $\Phi_{\hat{i}}$ is always well-defined and 1-1. We write $\Phi_M$ instead of~$\Phi_{\hat{i}}.$
Now, we investigate when a $p$-structure $\T$ for $\Sigma$ determines one for $\hM$, in other words, when $\T\in\Im\Phi_M$. If it is the case and the top cover of $\hM$ is $r$-torsion in $\Omega_3(B\Z_d)$, we can define 
$$\l_\T(M):=\l(\hM,(\Phi_M)^{-1}(\T)) \in \Z\Big[\frac{1}{r}\Big]\otimes_\Z L^0(\Q(\zeta_d)).$$
From now on, let $(\{\Sigma_{(t)}\}, \phi)$ be a $p$-structure $\T$ for $\Sigma$, and $F_{(t)}=\pi_1(\Sigma_{(t)})$.

\begin{lemma}
\label{lemma:defining condition}
Let $\T$ be a $p$-structure of height $h$ for $\Sigma$ of order $q$. For a homology cylinder $M$, $\T \in \Im\Phi_M$ if and only if $\tilde\mu_q^{\vphantom{}}(M)_k^{\vphantom{}} \in F_{(h+1)}/F_q$ for all~$k$.
\end{lemma}

\begin{proof}
For $0\leq t \leq h$, suppose there is $\hM_{(t)}$ corresponding to $\Sigma_{(t)}$ (and~$M_{(t)}$).
We shall prove that $\hM_{(t+1)}$ exists whose pullback cover is $\Sigma_{(t+1)}$ if and only if $\tilde \mu_q^{\vphantom{}}(M)_k^{\vphantom{}}$ is in $F_{(t+1)}/F_q$ for all~$k$. Since $\Sigma \to M$ is a $p$-tower map, we have $F_{(t)}\to \pi_1(M_{(t)}) \twoheadrightarrow \Gamma_{(t)}$ which determines $\Sigma_{(t+1)}$ and~$M_{(t+1)}$. All $\tilde \mu_q^{\vphantom{}}(M)_k^{\vphantom{}}$ are in $F_{(t+1)}/F_q$ if and only if all $\tilde\l(M)_k^{\vphantom{}}$ are in $\pi_1(M_{(t+1)})$, i.e.\ all $\tilde\l(M)_k^{\vphantom{}}$ vanish in $\Gamma_{(t)}$. It means that $\pi_1(M_{(t)})\twoheadrightarrow \Gamma_{(t)}$ factors through $\pi_1(\hM_{(t)})$, or equivalently, $\hM_{(t+1)}$ exists. The assumption is true for $t=0$, and this completes the proof by induction. 
\end{proof}

\begin{theorem} 
  \leavevmode \Nopagebreak 
  \label{theorem:p-tower map}
  Suppose $M$ is a homology cylinder.
  \begin{enumerate} 
	\item  $\hat{i}\colon \Sigma \to \hM$ is a $p$-tower map of order $q$ if and only if $\tilde\mu_q^{\vphantom{}}(M)$ vanishes.
	\item $\hat{i}\colon \Sigma \to \hM$ is a $p$-tower map if and only if $\tilde\mu^{\vphantom{}}(M)$ vanishes.
\end{enumerate}	
\end{theorem}

\begin{proof}
\begin{enumerate}
  \item By Lemma~\ref{lemma:defining condition}, $\Phi_M$ between $p$-structures of order $q$ is surjective if and only if $\tilde\mu_q^{\vphantom{}}(M)_k^{\vphantom{}} \in \bigcap_G G/F_q$ for all $k$, where $G/F_q$ ranges over all subgroups of $F/F_q$ such that the index $[F:G]$ is a power of $p$. Thus, it suffices to show that $\bigcap_G G/F_q$ is trivial.
From the fact that $F/F_q$ is a residually $p$-group \cite{Gr}, $\bigcap_G G/F_q$ is trivial.
\item This follows directly from (1).	\qedhere
\end{enumerate}
\end{proof}

From the proof of the theorem, we give a weakened condition for a map to be a $p$-tower map:
\begin{theorem}
  Let $X$ and $Y$ be connected CW-complexes having finitely generated fundamental groups.
  Suppose a map $f\colon X\to Y$ induces a surjection $f_{*}$ on~$\pi_1(-)/\pi_1(-)_q$.
  If $f_{*}$ is injective, then $f$ is a $p$-tower map of order~$q$.
  Moreover, if $\pi_1(X)/\pi_1(X)_q$ is a residually $p$-group, then the converse is also true.
\end{theorem}
We remark that this proposition can be applied to meridian maps of (string) links. This gives the affirmative answer to the question in \cite[Remark~6.4]{C10}: ``if $L$ is a link with vanishing $\bar\mu$-invariants, then is a meridian map into the surgery manifold of the link a $p$-tower map?''. Moreover, the converse is also true.

\subsection{Homology cobordism invariants}
\label{sec:homology cobordism invariants}
We study the homology cobordism invariance of~$\l_\T$.
\begin{proposition}
  \label{proposition:homology cobordism invariant}
   	Suppose homology cylinders $M$ and $M'$ are homology cobordant. Then $\l_\T(M)$ is defined as an element in $\Z[\frac{1}{r}]\otimes L^0(\Q(\zeta_d))$ if and only if $\l_\T(M')$ is defined as an element in $\Z[\frac{1}{r}]\otimes L^0(\Q(\zeta_d))$. In that case, $\l_\T(M)=\l_\T(M')$ in $\Z[\frac{1}{r}]\otimes L^0(\Q(\zeta_d))$.
\end{proposition}

Theorem~\ref{theorem:p-tower map} and Proposition~\ref{proposition:homology cobordism invariant} provide 
invariants of homology cobordism group of homology cylinders. More precisely, we have
$$\l_\T \colon \H(q) \to \Z\Big[\frac{1}{d}\Big]\otimes_\Z L^0 (\Q(\zeta_d))$$
for each $p$-structure $\T$ for $\Sigma$ of order $q$.

To prove the proposition, we construct a homology cobordism between $\hM$ and~$\hM'$:

\begin{lemma}
  \label{lemma:V}
	Suppose homology cylinders $(M,i_+^{\vphantom{}},i_-^{\vphantom{}})$ and $(M,i_+',i_-')$ are homology cobordant. Then there is a homology cobordism $\hW$ between $\hM$ and $\hM'$ such that $\Sigma \xrightarrow{\hat{i}}\hM \to \hW$ and $\Sigma \xrightarrow{\hat{i'}}\hM' \to \hW$ are homotopic.
\end{lemma}

\begin{proof}
Let $W$ be a homology cobordism between $M$ and~$M'$. We construct a 4-manifold $\hW$ from $W$ by identifying tubular neighborhoods of $i_+^{\vphantom{}}(\Sigma)$ and $i_-^{\vphantom{}}(\Sigma)$ in the boundary of~$W$. Then the boundary of $\hW$ is $\hM \cup \hM'$. We remind the reader that $E$ is the trivial homology cylinder for the next. Equivalently, $\hW$ is obtained by attaching $E\times I$ to the tubular neighborhood of $\partial M $ in the boundary of $W$ such that $\partial E \times  I$ = (tubular neighborhood of $\partial M$), $\partial E \times 0 \subset M$ and $\partial E\times 1\subset M'$. We need to check that the inclusion maps $\hM \hookrightarrow \hW$ and $\hM' \hookrightarrow \hW$ are homology equivalences. By comparing the Mayer-Vietoris sequences of $(M,E\times 0)$ and $(W,E \times I)$, we obtain an isomorphism $H_*(\hM) \to H_*(\hW)$ using the five lemma. Hence $\hM\hookrightarrow \hW$ is a homology equivalence, and similarly for $\hM'\hookrightarrow \hW$.
The homotopy conclusion in the statement follows from the construction of~$\hW$.
\end{proof}

\begin{proof}[Proof of Proposition~\ref{proposition:homology cobordism invariant}]
  Suppose $\l_\T(M)\in \Z[\frac{1}{r}]\otimes L^0(\Q(\zeta_d))$ is defined, i.e.\ there is a $p$-structure $\S$ for~$\hM$ such that $\Phi_M(\S) = \T$ and the top cover of $\hM$ is $r$-torsion in~$\Omega_3(B\Z_d)$. Since $\hM \hookrightarrow \hW$ and $\hM ' \hookrightarrow \hW$ are 2-connected, they are $p$-tower maps and so there is a $p$-structure $\S'$ for $\hM'$ corresponding to~$\S$ under the inclusion-induced bijections. From the homotopy conclusion in Lemma~\ref{lemma:V}, it follows that $\Phi_{M'}(\S') = \T$. Hence $\l_\T(M')$ is also defined.
By applying \cite[Theorem 3.1]{C10}, we obtain $\l_\T(M) = \l_\T(M')\in \Z[\frac{1}{r}]\otimes L^0(\Q(\zeta_d))$.
\end{proof}
 
   For later use, we recall two key ingredients of the proof of \cite[Theorem 3.1]{C10}:
\begin{enumerate}
	\item[(K1)] Suppose $X\supset Y$ are CW complexes with finite $n$-skeletons. If $H_i(X,Y;\zp) = 0$ for $i \leq n$, then $H_i(\tilde X, \tilde Y; \zp) = 0$ for $i \leq n$ where $\tilde X$ is a $p$-cover of $X$ and $\tilde Y$ is the pullback cover of $Y$ by the inclusion $Y \hookrightarrow X$ \cite[Lemma 3.3]{C10}.
	\item[(K2)] If a 4-manifold $X$ satisfies
	$$\Im\{H_2(\tilde X;\Q) \to H_2(\tilde X,(\text{a subspace of) } \partial\tilde X;\Q)\} = 0$$
for a $\Z_d$-cover $\tilde X$ of $X$, then the Witt class of the $\Q(\zeta_d)$-coefficient intersection form $[\l_{\Q(\zeta_d)}(X)] $ vanishes.
\end{enumerate}
As in the proof of \cite[Theorem~2.4]{C09}, to prove (K2), we use the fact that $\Q(\zeta_d)$ is flat over $\Q[\Z_d]$; by the universal coefficient theorem,
\begin{align*}
	H_2(X;\Q(\zeta_d)) \; & = \; H_2(X;\Q[\Z_d]) \otimes_{\Q[\Z_d]}\Q(\zeta_d) \\
	& = \; H_2(\tilde{X} ; \Q) \otimes_{\Q[\Z_d]}\Q(\zeta_d).
\end{align*}
Similarly for a subspace of~$\partial\tilde{X}$.
Thus the condition implies
$$\Im\{H_2(X;\Q(\zeta_d)) \to H_2(X,\partial X; \Q(\zeta_d))\} = 0,$$
and $[\l_{\Q(\zeta_d)}(X)] = 0.$

\subsection{Homomorphism of homology cobordism groups}
  \label{sec:additivity condition}
In this subsection, we investigate the additivity of $\l_\T$.
For $M, N$ in $\H(q)$ and a $p$-structure $\T$ for $\Sigma$ of order $q$,
we need to check whether $\l_\T(M)+ \l_\T(N)  - \l_\T (M\cdot N)=0$. 
We will construct a cobordism $V$ between $\hM \cup \hN$ and $\widehat{M \cdot N}$ such that if $M$ and $N$ are in $\H(q)$, then the inclusions from $\hM$, $\hN$, and $\hMN$ into $V$ are $p$-tower maps of order~$q$. After then, we will investigate when the difference of the Witt classes of two intersection forms, twisted and untwisted, of the top cover of $V$ vanishes. 

	For our purpose, we can use a``standard'' cobordism $V$ as in~\cite{C09, CHH}. The cobordism $V$ is obtained from $\big(\widehat{(M,i_+^{\vphantom{}},i_-^{\vphantom{}})} \cup \widehat{(N,j_+^{\vphantom{}},j_-^{\vphantom{}})}\big) \times I$ by identifying product neighborhoods of $\hat{i}(\Sigma)$ and $\hat{j}(\Sigma)$ in $(\hM \cup \hN) \times 1$. Since the identification can be thought of as attaching $\Sigma \times I \times I$ along the product neighborhoods, $V$ can be obtained by attaching one 1-handle which connects $\hM \times I$ to $\hN \times I$ and $(2g+n-1)$ 2-handles along simple closed curves on $\hat{i}(\Sigma)\#_b ~\hat{j}(\Sigma)$ in $\hM \# \hN$ corresponding to $\hat{i}_*(z)\cdot \hat{j}_*(z^{-1})$ (up to homotopy) where $z$ ranges over disjoint simple closed curves representing $x_i^{\vphantom{}}$, $m_j^{\vphantom{}}$ and $l_j^{\vphantom{}}$ in Figure~\ref{figure:Sigma}.
($\#$ denotes connected sum and $\#_b$ denotes boundary connected sum.)
Then, by construction, we have the following property, which we state as a lemma:
\begin{lemma}
  \label{lemma:cobordism}
  For any homology cylinders $M$ and $N$ over $\Sigma$, there is a cobordism $V$ between $\hM \cup \hN$ and $\hMN$ such that the following diagram is commutative up to homotopy rel $*$. 
     $$\begin{diagram} 
  \dgARROWLENGTH=.6\dgARROWLENGTH
    \node{} \node{\hM} \arrow{se,J} \\
    \node{\Sigma} \arrow{ne,l}{\hat{i}} \arrow{se,r}{\begin{substack}{i_+^{\vphantom{}} = j_-^{\vphantom{}}
    }\end{substack}}\arrow{e,t}{\hat{j}} \node{\hN} \arrow{e,J} \node{V} \\
    \node{} \node{\hMN} \arrow{ne,J}
  \end{diagram} $$
\end{lemma}

Note that $\hMN$ and $\widehat{N\cdot M}$ are homeomorphic, and the marking for $\hMN$ and the marking for $\widehat{N\cdot M}$ induce the same homomorphism $\pi_1(\Sigma)\to \pi_1(\hMN)\cong \pi_1(\widehat{N\cdot M})$. 
  
\begin{lemma}
   If $M, N \in \H(q)$, then all maps in the above diagram induce isomorphisms on~$\pi_1(-)/\pi_1(-)_q$. Consequently, all the maps are $p$-tower maps of order~$q$.
\end{lemma}

\begin{proof}
It suffices to prove that the embedding $\hMN \hookrightarrow V$ induces an isomorphism on $\pi_1(-)/\pi_1(-)_q$.
Both $\pi_1(\hMN)$ and $\pi_1(V)$ are quotient groups of $\pi_1(M)\ast\pi_1(N)$.
The group $\pi_1(\hMN)$ is the quotient by the normal subgroup generated by $(i_-^{\vphantom{}})_*(z) (j_+^{\vphantom{}})_*(z^{-1})$ for $z\in F$ and $\tilde\l(M)_k^{\vphantom{}} \tilde\l(N)_k^{\vphantom{}}$ for all $k$.
The group $\pi_1(V)$ is the quotient by the normal subgroup generated by $(i_-^{\vphantom{}})_*(z) (j_+^{\vphantom{}})_*(z^{-1})$ for $z\in F$ and $\tilde\l(M)_k^{\vphantom{}}$, $\tilde\l(N)_k^{\vphantom{}}$ for all $k$.
Hence if $M$ is in $\H(q)$ then $\tilde\l(M)_k^{\vphantom{}}$ is in $(\pi_1(M)\ast\pi_1(N))_q$. It follows that the claim is true. 
\end{proof}

For corresponding $p$-towers for $\hM$, $\hN$, $\hMN$, and $\hW$, we have $\partial V_{(t)}=\hM_{(t)}\cup \hN_{(t)} \cup -\hMN_{(t)}$, and hence $$\l_\T(M)+\l_\T(N)-\l_\T(M\cdot N) = [\l_{\Q(\zeta_d)}(V)] - [\l_\Q(V)].$$

The following theorem presents sufficient conditions for the intersection forms on the right hand side to be Witt trivial. Recall that by Lemma~\ref{lemma:defining condition}, if $M \in H(q)$ and $\T$ is a $p$-tower for $\Sigma$ of order $q$, then $\tilde\mu_q^{\vphantom{}}(M)_k^{\vphantom{}}\in F_{(t)}/F_q$ for all~$t$. Also the $k$th coordinate $\tilde\mu(M)_k^{\vphantom{}}$ of $\tilde\mu(M)$ lives in $\varprojlim_{s\geq q} F_{(t)}/F_s$. We will consider a map $\varprojlim_{s\geq q} F_{(t)}/F_s \to \varprojlim_s H_1(F_{(t)})\otimes \Z_{p^s}$, which will be defined in the proof of Theorem~\ref{theorem:additivity condition} below.

\begin{theorem}
	\label{theorem:additivity condition}
	Suppose $\T$ is a $p$-structure of height $h$ for $\Sigma$ of order $q$ and $M$ is a homology cylinder in~$\H(q)$. Then the following are equivalent:
	\begin{enumerate}
		\item[(C1)] $H_1(\Sigma_{(t)};\Z) \to H_1(\hM_{(t)}; \Z)$ is injective.
		\item[(C2)] $H_1(\Sigma_{(t)};\zp) \to H_1(\hM_{(t)};\zp)$ is injective.
		\item[(C3)] $H_1(\Sigma_{(t)};\Q) \to H_1(\hM_{(t)};\Q)$ is injective.
		\item[(C4)] All $\tilde \l(M)_k^{\vphantom{}}$ are torsion elements in~$H_1(M_{(t)})$.
		\item[(C5)] All $\tilde\mu(M)_k^{\vphantom{}}$ lie in the kernel of $\varprojlim_{s\geq q} F_{(t)}/F_s \to \varprojlim_s H_1(F_{(t)})\otimes \Z_{p^s}.$
	\end{enumerate}
If $M$ and $N$ are in $\H(q)$ and either $M$ or $N$ satisfies \textup{(}one of\textup{)} \textnormal{(C1)--(C5)} for $t=h+1$, then
$$\l_\T(M) + \l_\T(N) = \l_\T(M\cdot N).$$	
In addition, the homology cylinders satisfying \textup{(}one of\textup{)} \textnormal{(C1)--(C5)} form a subgroup of~$\H(q)$.
\end{theorem}
We remark that in (C2) and (C3), the injectivity of the map implies that it is an isomorphism. See the proof below.
\begin{proof}
	The implications (C1) $\Rightarrow$ (C2) $\Rightarrow$ (C3) follow that $\zp$ is flat over $\Z$ and $\Q$ is flat over $\zp$. From (K1) in the last paragraph in Section~\ref{sec:homology cobordism invariants}, $H_1(\Sigma_{(t)};\zp) \xrightarrow{\cong} H_1(M_{(t)};\zp)$. Since $\pi_1(M_{(t)})\to\pi_1(\hM_{(t)})$ is surjective, 
	the homomorphisms in (C2) and (C3) are always surjective. Since $H_1(\Sigma_{(t)})$ is torsion-free, $H_1(\Sigma_{(t)};\Z) \to H_1(\Sigma_{(t)};\Q)$ is injective, and so (C3) implies (C1). Also, since $H_1(\Sigma_{(t)})$ is torsion-free, (C3) and (C4) are equivalent. 
	
	(C4) says that all $\tilde \l_k^{\vphantom{}}$ are $0$ in $H_1(M_{(t)};\zp)$ since $H_1(M_{(t)};\zp)$ has no $p$-torsion. Since $$\Ker\{\pi_1(M_{(t)}) \to H_1(M_{(t)};\zp)\} = \bigcap_s \Ker\{\pi_1(M_{(t)}) \to H_1(M_{(t)};\Z_{p^s})\},$$ (C4) is equivalent to $$\tilde\l_k^{\vphantom{}}\in\Ker\{\pi_1(M_{(t)}) \to H_1(M_{(t)};\Z_{p^s})\}$$ for all~$s$. Because $H_1(M_{(t)};\Z_{p^s})$ is a finite $p$-group, by Lemma \ref{lemma:algebra}, the kernel contains $\pi_1(M)_{q_s}$ for some~$q_s\geq q$. The lift $\Sigma_{(t)}\to M_{(t)}$ of $i_+^{\vphantom{}}$ induces isomorphisms on $\pi_1(-)/\pi_1(-)_{q_s}$ and on $H_1(-;\Z_{p^s})$. Thus (C4) is also equivalent to $$\tilde \mu_{q_s}^{\vphantom{}}(M) \in \Ker\{\pi_1(\Sigma_{(t)})/\pi_1(\Sigma)_{q_s} \to H_1(\Sigma_{(t)}; \Z_{p^s})\}$$ for all~$s$. Taking inverse limit, (C5) is obtained.
	
	Note that if $M$ satisfies one of (C1) to (C5) for $t = h+1$, then it is also true for $t= h$ by considering (C4) with $H_1(M_{(h+1)}) \to H_1(M_{(h)})$. We claim that $$\Im\{H_2(V_{(t)};\Q) \to H_2(V_{(t)},\hM_{(t)}\cup \hN_{(t)};\Q)\} = 0$$ for $t=h,h+1$. If so, by (K2) in the last paragraph of Section~\ref{sec:homology cobordism invariants}, $[\l_\Q(V)]$ and $[\l_{\Q(\zeta_d)}(V)]$ vanish. Hence $\l_\T(M) + \l_\T(N) - \l_\T(M\cdot N) = 0$. The cobordism $V$ can be considered as a union of $(\hM \cup \hN) \times I$ and $\Sigma \times I \times I$ whose intersection is $(\hat{i}(\Sigma) \cup \hat{j}(\Sigma))\times 1.$ Applying the Mayer-Vietoris theorem, we obtain an exact sequence
	$$H_2(\hM_{(t)} \cup \hN_{(t)};\Q) \to H_2(V_{(t)} ; \Q) \to H_1(\Sigma_{(t)};\Q) \to H_1(\hM_{(t)} \cup \hN_{(t)};\Q).$$
We have the injectivity of the rightmost map for $t\leq h+1$. Then, the leftmost map is surjective, and the claim is shown.
\end{proof}

We remark that if one of (C1)--(C5) holds for $h$, then $H_1(\hM_{(h)})$ is $p$-torsion free.
Thus in that case, $\l_\T(M)$ lives in $L^0(\Q(\zeta_d))$, by Lemma~\ref{lemma:p-torsion-free}.

Applying the theorem, we obtain a sufficient condition for $\l_\T$ to be a homomorphism of $\H(q)$:

\begin{corollary}
	Suppose $\T$ is a $p$-structure of height $h$ for $\Sigma$ of order~$q$.
	If $\T$ satisfies $F_q \subset [F_{(h+1)}, F_{(h+1)}]$, then 
	 $$\l_\T\colon \H(q)\to L^0(\Q(\zeta_d))$$ is a homomorphism on~$\H(q)$.
\end{corollary}

\begin{proof}
	From the hypothesis, we have $F_{(h+1)}/F_q \to H_1(F_{(h+1)})$. This induces a homomorphism $F_{(h+1)}/F_q \to \varprojlim_s H_1(F_{(h+1)})\otimes \Z_{p^s}$, and $\varprojlim_s F_{(h+1)}/F_s \to \varprojlim_s H_1(F_{(h+1)})\otimes \Z_{p^s}$ factors through it. If $M$ is in $\H(q)$, then all $\tilde\mu_q(M)_k^{\vphantom{}}$ vanish in $F_{(h+1)}/F_q$, and $M$ satisfies (C5) in Theorem~\ref{theorem:additivity condition} for $h+1$.
\end{proof}

Also, appealing to (C5), we derive a main result as another corollary:
\begin{corollary}
    \label{cor:homomorphism}
	For any $p$-structure $\T$ for $\Sigma$, 
	$$\l_\T\colon \H(\infty) \to L^0(\Q(\zeta_d))$$ 
	is a homomorphism.
\end{corollary}

\begin{proof}[Proof of Lemma~\ref{lemma:algebra}] 
  Let $N_{(0)}=G$ and $$N_{(t+1)}=\Ker\{N_{(t)}\to \frac{N_{(t)}}{[N_{(t)},N_{(t)}]}=H_1(G;\Z[G/N_{(t)}])\to H_1(G;\Z_{p^{a_t}}[G/N_{(t)}])\}$$ where $p^{a_t}=|G_{(t)}/G_{(t+1)}|$. That is, $N_{(t)}=\P^t G$, the $\P$-mixed-coefficient commutator series, where $\P=(\Z_{p^{a_0}},\Z_{p^{a_1}},\ldots)$ (see~\cite{C}). We claim:
  \begin{enumerate}
    \item $N_{(t)} \vartriangleleft G$,
    \item $G/N_{(t)}$ is a finite $p$-group, and
    \item $N_{(t)} \subset G_{(t)}$.
  \end{enumerate}
  From (2), $G/N_{(t)}$ is nilpotent.
  Therefore, for all $t$, there is some $q$ such that $G_q \subset N_{(t)}$. Combining this with (3), we obtain the conclusion.

  Let us show the above three claims.
  (1) can be shown by induction since $G$ acts on $H_1(G;\Z_{p^{a_t}}[G/N_{(t)}])$ by congugation.
  Because $G$ is finitely generated and $H_1(G;\Z_{p^{a_t}}[G/N_{(t)}])$ is a finite $p$-group, each $N_{(t)}/N_{(t+1)}$ is a finite $p$-group and so is $G/N_{(t)}$.
  We use an induction for (3).
  $$\begin{diagram}  \dgARROWLENGTH=.5em 
    \node{N_{(t+1)}} \arrow{e,t,J}{\Ker}\arrow[2]{s,..,J} \node{N_{(t)}} \arrow[2]{ee}\arrow{se,r}{ab.}\arrow[2]{s,J} \node[2]{H_1(G;\Z_{p^{a_t}}[G/N_{(t)}])} \arrow[2]{s,..}\\
    \node[3]{H_1(G;\Z[G/N_{(t)}])} \arrow{ne}\arrow{se} \\
    \node{G_{(t+1)}} \arrow{e,t,J}{\Ker} \node{G_{(t)}} \arrow[2]{e} \node[2]{G_{(t)}/G_{(t+1)}}
  \end{diagram}$$
  In the above diagram, $H_1(G;\Z_{p^{a_t}}[G/N_{(t)}]) = H_1(N_{(t)})/ p^{a_t} H_1(N_{(t)})$ and $G_{(t)}/G_{(t+1)}$ is abelian and of order $p^{a_t}$, and so the rightmost vertical map exists. Hence the leftmost vertical map also exists and is injective.
\end{proof}

\section{Structures in $\H(\infty)$ and its subgroups}
  \label{sec:effect}
In this section, we construct infinitely many homology cylinders to investigate the structure of $\H$, especially in~$\H(\infty)$. Our examples are constructed by infection on the trivial homology cylinder. We start with a description of infection by a knot. 
For a 3-manfold $M$ and a simple closed curve $\alpha$ in the interior of $M$, by removing an open tubular neighborhood of $\alpha$ from $M$ and by filling in it with the exterior of a knot $K$ in $S^3$ so that the meridian and the preferred longitude of $K$ are identified with the preferred longitude and the meridian of $\alpha$, respectively, we obtain a new 3-manifold~$N$. We say that $N$ is obtained from $M$ by \emph{infection along $\alpha$ using}~$K$. This construction appeared in \cite{COT, COT2}.

Let $(M,i_+^{\vphantom{}},i_-^{\vphantom{}})$ be a homology cylinder and $M'$ be obtained from $M$ by infection along $\alpha$ using~$K$. It is well known that there is a homology equivalence $f\colon M' \to M$, which extends the identity map between $M' -$ (exterior of $K$) and $M -$ (tubular neighborhood of $\alpha$) (for example, see \cite[Proposition 4.8]{C10}). Hence $M'$ is a homology cylinder with markings $i_\pm'\colon \Sigma \to M'$ induced by~$i_\pm^{\vphantom{}}$.

We consider the effect of infection on the invariants of Garoufalidis and Levine \cite{GL}, Cha, Friedl, and Kim \cite{CFK}, Morita \cite{M}, Sakasai \cite{S}, Cochran, Harvey, and Horn \cite{CHH} (see the introduction), the extended Milnor invariants $\tilde\mu_q^{\vphantom{}}$, and the Hirzebruch-type invariants~$\l_\T$.
Let $H=H_1(\Sigma)$.
\begin{enumerate}
\item[(a)] Garoufalidis-Levine homomorphisms~\cite{GL} 
$$\eta_q^{\vphantom{}}\colon\H_{g,n} \to \Aut(F/F_q)$$ \\
    The left commutative diagram induces the right commutative diagram.
    $$
    \begin{diagram}  \dgHORIZPAD=.6em
      \node{\Sigma} \arrow{e,t}{i_-'} \arrow{se,b}{i_-^{\vphantom{}}} \node{M'} \arrow{s,l}{f} \node{\Sigma}  \arrow{w,t}{i_+'} \arrow{sw,b}{i_+^{\vphantom{}}} \\
      \node{} \node{M}
    \end{diagram}
    \qquad
        \begin{diagram}  \dgHORIZPAD=.7em
      \node{\frac{\pi_1(\Sigma)}{\pi_1(\Sigma)_q}} \arrow{se,tb}{\cong}{(i_-^{\vphantom{}})_{*q}^{\vphantom{}}} \arrow{e,tb}{(i_-')_{*q}^{\vphantom{}}}{\cong} \node{\frac{\pi_1(M')}{\pi_1(M')_q}} \arrow{s,r}{f_*} \node{\frac{\pi_1(\Sigma)}{\pi_1(\Sigma)_q}} \arrow{w,tb}{(i_+')_{*q}^{\vphantom{}}}{\cong} \arrow{sw,tb}{\cong}{(i_+^{\vphantom{}})_{*q}^{\vphantom{}}}\\
      \node{} \node{\frac{\pi_1(M)}{\pi_1(M)_q}}
    \end{diagram}
    $$
Since $\eta_q^{\vphantom{}}(M)=(i_+^{\vphantom{}})_{*q}^{-1} \circ (i_-^{\vphantom{}})_{*q}^{\vphantom{}}$, $\eta_q^{\vphantom{}}(M')=\eta_q^{\vphantom{}}(M)$.

\item[(b)] Morita homomorphism~\cite{M}
$$\tilde{\rho}\colon\H_{g,1} \to (\Lambda^3H_\Q \oplus \bigoplus_{k=1}^\infty S^{2k+1} H_\Q)\rtimes Sp(2g,\Q)$$
 Here $S^{2k+1}H$ denotes the ($2k+1$)st symmetric power of $H$ and $H_\Q=H\otimes \Q$.
This $\tilde{\rho}$ is the composition of the limit of $\eta_q^{\vphantom{}}$ with a trace map. Hence $\tilde{\rho}(M')=\tilde{\rho}(M)$ by (a), and $\bigcap_{q=1}^\infty \Ker(\eta_q^{\vphantom{}}) \subset \Ker(\tilde{\rho})$.

\item[(c)] Sakasai's Magnus representations~\cite{S}
$$r_{q}\colon\H_{g,1} \to GL(2g,Q(F/F_q)) \; \textrm{ and } \; r\colon\H_{g,1} \to GL(2g,\Lambda_{\hF})$$ 
Here $Q(F/F_q):=\Z[F/F_q](\Z[F/F_q]-\{0\})^{-1}$, $\hF$ is the algebraic closure with respect to $\Z$ \cite{C08} (which is called the acyclic closure by Sakasai), and $\Lambda_{\hF}$ is the Cohn localization of the augmentation map $\Z[\hF] \to\Z$.
They are crossed homomorphisms, and the restrictions to $\H_{g,1}[q]$ and $\Ker\{\H_{g,1}\to \Aut(\hF)\}$ are homomorphisms, respectively.
  Similar to (a), by taking the first relative homology $H_1(-,\ast;Q(F/F_q))$ on the left diagram in (a), we see that $r_q(M')=r_q(M)$. Also by taking the acyclic closure of the fundamental group, we obtain $r(M)=r(M')$. 
  
\item[(d)] Cha-Friedl-Kim's torsion invariant~\cite{CFK} 
$$\tau \colon \H_{g,n} \to Q(H)^\times / \pm HAN$$ 
Here $Q(H)$ is the quotient field of $\Z[H]$, $$A=\{p^{-1}\cdot \eta (p)~|~p\in Q(H)^\times, \eta \in \Im\eta_2^{\vphantom{}} \},~~ N=\{q \cdot \bar{q} ~|~ q \in Q(H)^\times\},$$ and $\bar{\hphantom{q}}$ is the extension of the involution of the group ring~$\Z[H]$. This is a homomorphism.
The effect of infection is studied in \cite[Theorem~4.2]{CFK}. We discuss some details for the reader's convenience.
  We first consider the effect on $\tau \colon \C_{g,n} \to Q(H)^\times / \pm H$. Let $\tau^\phi_K$ be the torsion of the acyclic cellular chain complex $C_*(S^3-K, m_K;\Z[H])$ with a meridian $m_K$ of~$K$. Let $\phi\colon H_1(S^3-K) \to H_1(M') \xrightarrow{(i_+')_*^{-1}} H$ be induced by the inclusion $S^3 - K \hookrightarrow M'$. Denote the tubular neighborhood of $\alpha$ by~$\nu(\alpha)$. Then we have
  $$\tau(M')\doteq \tau(M) \cdot \tau(S^3-K) \cdot \tau(\nu(\alpha))^{-1}.$$
  Here $\doteq$ means the equality in $Q(H)^\times/\pm H$. From the exact sequence
  $$0 \to C_*(m_K;Q(H)) \to C_*(S^3-K;Q(H)) \to C_*(S^3-K, m_K;Q(H)) \to 0 $$
  with $(S^3-K,m_K)$ acyclic, we obtain $\tau(S^3-K)\doteq \tau(m_K)\cdot \tau^\phi_K$. The class of $m_K$ in $H_1(S^3-K)$ maps to the class of $\alpha$ in $H_1(M')$, and $S^1$ and $S^1\times D^2$ are simple homotopy equivalent. Thus $\tau(m_K)\doteq \tau(\alpha)\doteq \tau(\nu(\alpha))$. From this, we obtain $\tau(M') \doteq \tau(M)\cdot \tau^\phi_K$.
  \begin{remark}
    \label{remark:torsion}
    Since any knot exterior in $S^3$ is a homology cylinder over~$\Sigma_{0,2}$, $$\overline{\tau(S^3-K)}\doteq \tau(S^3-K) \in Q(H_1(\Sigma_{0,2}))^\times/\pm H_1(\Sigma_{0,2})$$ by \cite[Lemma 3.13]{CFK}, and the inclusion $S^3-K \hookrightarrow M'$ induces $$\overline{\tau(S^3-K)}\doteq \tau(S^3-K) \in Q(H)^\times/\pm H.$$ Thus, $\overline{\tau(M')}\doteq \tau(M')$ if $\overline{\tau(M)}\doteq \tau(M)$.
  \end{remark}

\item[(e)] Cochran-Harvey-Horn's von Neumann $\rho$-invariants~\cite{CHH}
$$\rho_q\colon \H_{g,1}[q] \to \R$$
  They studied the effect of infection in \cite[Proposition 8.11]{CHH}: when $\alpha \in \pi_1(\hM)_k$ but no power of $\alpha$ lies in $\pi_1(\hM)_{k+1}$, 
  $$ \rho_q(M')-\rho_q(M) = \left\{ \begin{array}{ll} 0 &  \textrm{if  } 2\leq q \leq k \\ \int_{S^1} \sigma_{K}(\omega)\, d\omega & \textrm{if  } q> k \end{array} \right. .$$
  Here $\sigma_{K}(\omega)$ is the Levine-Tristram signature of~$K$. The map $\rho_q$ is a ``quasimorphism'' on $H_{g,1}[q]$, and it is a homomorphism on~$\Ker r_q$.

\item[(f)] Extended Milnor invariants $$\tilde\mu_q\colon \H_{g,n} \to  (F/F_q)^{2g+n-1}$$ 
Their restrictions are homomorphism on $\H[q]$ or~$\H^0[q-1]$.  
Since they are also obtained from the induced maps on $\pi_1(-)/\pi_1(-)_q$, they are preserved by infection.
\end{enumerate}

Now we consider the Hirzebruch-type invariants
 $$\l_\T\colon \H_{g,n}(q) \to \Z\Big[\frac{1}{d}\Big] \otimes_\Z L^0(\Q(\zeta_d))$$
 
defined in Section~\ref{sec:hirzebruch}.
Here $\T$ is a $p$-structure of order $q$. It is a homomorphism on the subgroup of homology cylinders satisfying (C1)--(C5) in Theorem~\ref{theorem:additivity condition} into $L^0(\Q(\zeta_d))$, especially on $\H_{g,n}(\infty)$.  

In \cite[Corollary 4.7]{C10}, the effect of infection is analyzed for general closed 3-manifolds. By applying it to homology cylinders, we obtain the following theorem:
\begin{theorem}[A special case of Corollary 4.7 in \cite{C10}]
  \label{theorem:effect}
  Let $M$ be a homology cylinder in $\H(q)$ and $\T$ be a $p$-structure of height $h$ for $\Sigma$. Let $\psi\colon \pi_1(M_{(h)})\to\Z_d$ be the character induced by $\T$. Let $\tilde\alpha_1, \tilde\alpha_2, \ldots\subset M_{(h)}$ be the components of the pre-image of $\alpha$ and $r_j$ be the degree of the covering map $\tilde\alpha_j \to \alpha$. Let $A$ be a Seifert matrix of $K$.
Then 
$$ \l_\T(M') = \l_\T(M) + \sum_j\big([\l_{r_j}(A,\zeta_d^{\psi([\tilde\alpha_j])})]-[\l_{r_j}(A,1)]\big)$$
where $[\l_r(A,\omega)]$ is the Witt class of (the nonsingular part of) the hermitian form represented by the following $r\times r$ block matrix:
$$\l_r(A,\omega) = 
  \begin{bmatrix}
    \vphantom{\ddots} A+A^T & -A & & & -\omega^{-1} A^T\\
    \vphantom{\ddots} -A^T & A+A^T & -A \\
    \vphantom{\ddots} & -A^T & A+A^T & \ddots \\
    \vphantom{\ddots} & & \ddots & \ddots & -A \\
    \vphantom{\ddots} -\omega A & & & -A^T & A+A^T
  \end{bmatrix}_{r\times r} .
  $$
  For $r=1,2$, $\l_r(A,\omega)$ should be understood as 
  $$
  \begin{bmatrix}
    (1-\omega)A+(1-\omega^{-1})A^T
  \end{bmatrix}
  \quad\text{and}\quad
  \begin{bmatrix}
    A+A^T & -A-\omega^{-1}A^T \\
    -A^T-\omega A & A+A^T
  \end{bmatrix}.
  $$
\end{theorem}

Let $E(\alpha,K)$ be the homology cylinder obtained from the trivial homology cylinder $E$ by infection using $K$ along~$\alpha$. 
By Remark~\ref{remark:torsion}, $\tau(E(\alpha,K)^2)\in Q(H)^\times/\pm HAN$ vanishes for any $\alpha$ and~$K$. 
For $K$ with $\int_{S^1} \sigma_{K}(\omega) = 0$, $\rho_q(E(\alpha, K))$ vanishes and all invariants in (a)--(f) vanish on $E(\alpha, K)^2$. 

We will choose a simple closed curve $\alpha$ and an infinite sequence $\{K_i\}$ of knots such that $E(\alpha,K_i)^2$ are distinguished by $\l_\T$. For this purpose, we need the following two lemmas:
\begin{lemma}[Lemma 5.3 in \cite{C09}]
  \label{lemma:infection-curve-and-tower}
  When $b_1(\Sigma)>1$, for any $h$, there exist a loop $\gamma$ in $\Sigma$ and a $p$-tower $\{\Sigma_{(t)}\}$ of height $h$ for $\Sigma$ satisfying the
  following:
  \begin{enumerate}
  \item $[\gamma] \in \pi_1(\Sigma)^{(h)}$.
  \item Every lift $\tilde \gamma_j$ of $\gamma$ in $\Sigma_{(h)}$ is a loop.
  \item There is a map $\phi\colon \pi_1(\Sigma_{(h)}) \to \Z$ which sends (the
    class of) each $\tilde\gamma_j$ to $-1$, $0$, or $1$ and sends at
    least one $\tilde\gamma_j$ to~$1$.
  \end{enumerate}
\end{lemma}

\begin{lemma}[Lemma 5.2 in \cite{C09}]
  \label{lemma:knots}
  For any prime $p$, there is an infinite sequence $\{K_i\}$ of knots together with a
  strictly increasing sequence $\{d_i\}$ of powers of $p$ satisfying
  the following properties:
  \begin{enumerate}
  \item $\sigma_{K_i}(\zeta_{d_i}) > 0$, and if $p=2$ then
    $\sigma_{K_i}(\zeta_{d_i}^s) \geq 0$ for any~$s$.
  \item If $i>j$ then $\sigma_{K_i}(\zeta_{d_j}^s) = 0$ for any $s$.
  \item $\int_{S^1} \sigma_{K_i}(\omega)\, d\omega = 0$.
  \item $K_i$ has vanishing Arf invariant.
  \end{enumerate}
\end{lemma}

We remark that (1) in Lemma~\ref{lemma:infection-curve-and-tower} and (3), (4) in Lemma~\ref{lemma:knots} will be used in Section~\ref{sec:cobordisms}. 

Now we obtain one of our main results: 

\begin{theorem}
  \label{theorem:infinite rank}
  Suppose $b_1(\Sigma)>1$. Then the abelianization of the intersection of the kernels of the invariants in \textup{(}a\textup{)--(}f\textup{)} is of infinite rank.
\end{theorem}

\begin{proof}
Let $\alpha$ be a simple closed curve obtained by pushing $i_+^{\vphantom{}} \circ \gamma$ into the interior of the trivial homology cylinder $E$,
 and $K_i$ be knots as in Lemma~\ref{lemma:knots}.

Let $\{\Sigma_{(t)}\}$ be the $p$-tower in Lemma~\ref{lemma:infection-curve-and-tower} and $\phi_d\colon \pi_1(\Sigma_{(h)}) \to \Z_d$ be the composition of the map $\phi\colon \pi_1(\Sigma_{(h)}) \to \Z$ in Lemma~\ref{lemma:infection-curve-and-tower} with the projection $\Z \to \Z_d$ which sends $1\in\Z$ to $1\in\Z_d$. Note that $\pi_1(\Sigma_{(t)})$ and $\pi_1(E_{(t)})$ are isomorphic. For $\T = (\{\Sigma_{(t)}\},\phi_d)$, due to Theorem~\ref{theorem:effect}, we have
$$\l_\T(E(\alpha,K_i)) = \sum_j \big([\l_1(A_i,\zeta_d^{\phi_d([\tilde{\gamma}_j])})] - [\l_1(A_i,1)]\big)$$
where $A_i$ is a Seifert matrix of~$K_i$. Observe that $\sign\l_1(A_i, \omega)=\sigma_{K_i}(\omega)$, $\sigma_{K_i}(1)=0$, and $\sigma_{K_i}(\omega)=\sigma_{K_i}(\omega^{-1})$. By the choice of $\phi_d$ from $\phi$,
$$\sign \l_\T(E(\alpha,K_i))=c \cdot \sigma_{K_i}(\zeta_d)$$
where $c$ is the number of lifts $\tilde{\gamma}_j$ sent to $\pm 1$ by $\phi\colon \pi_1(\Sigma_{(h)}) \to \Z$. Note that $c>0$.

Now we prove that $E(\alpha, K_i)^2$ are linearly independent in the abelianization. Suppose they are linearly dependent. Then $\sum_i a_i E(\alpha,K_i)^2 = 0$ in the abelianization where not all $a_i$ are zero. Let $i_0$ be the smallest integer such that $a_{i_0}\neq 0$. Let $d_i$ be the number in Lemma~\ref{lemma:knots}, and consider the $p$-structure $\T=(\{\Sigma_{(t)}\},\phi_{d_{i_0}})$. Then
\begin{align*}
0=\sign \l_\T\Big(\sum_i~ a_i~ E(\alpha,K_i)^2\Big) &= 2~\sum_i~ a_i ~\sign \l_\T\big(E(\alpha,K_i)\big) \\
&= 2~\sum_i ~a_i~ c ~ \sigma_{K_i}(\zeta_{d_{i_0}}) \\
&= 2~\sum_{i\geq i_0}~ a_i ~c ~ \sigma_{K_i}(\zeta_{d_{i_0}})  \\
&= 2~a_{i_0}~ c ~ \sigma_{K_{i_0}}(\zeta_{d_{i_0}})    \\
&\neq 0.
\end{align*}
This contradiction implies the linear independence of $E(\alpha,K_i)^2$ in the abelianization of the intersection of those kernels. Therefore abelianization has infinite rank.
\end{proof}

We note that the homology cylinders $E(\alpha, K_i)^2$ used in the proof of Theorem~\ref{theorem:infinite rank} are boundary homology cylinders by the following lemma: 
\begin{lemma}
  \label{lemma:preserving}
  A homology cylinder obtained by infection from a boundary homology cylinder, an $\hF$-homology cylinder, or a homology cylinder with vanishing $\tilde\mu_q^{\vphantom{}}$ is a boundary homology cylinder, an $\hF$-homology cylinder, or a homology cylinder with vanishing $\tilde\mu_q^{\vphantom{}}$, respectively.
\end{lemma}

\begin{proof}
  Let $M'$ be a homology cylinder obtained from a homology cylinder $M$ by infection using $K$ along~$\alpha$. There is a map $f\colon M' \to M$ which extends the identity between $M'-$ (exterior of $K$) and $M-$ (tubular neighborhood of $\alpha$). Especially $f$ extends the identity between boundaries. The induced map $f_*$ on the fundamental groups sends $\tilde\l(M')_k^{\vphantom{}}$ to $\tilde\l(M)_k^{\vphantom{}}$ for each~$k$. Since $f$ induces isomorphisms on $\widehat{\pi_1(-)}$ and $\pi_1(-)/\pi_1(-)_q$, $\tilde\l(M)_k^{\vphantom{}}$ vanishes in $\widehat{\pi_1(M)}$ or $\pi_1(M)/\pi_1(M)_q$ if and only if $\tilde\l(M')_k^{\vphantom{}}$ vanishes in $\widehat{\pi_1(M')}$ or~$\pi_1(M')/\pi_1(M')_q$. 
  If there is a splitting $\phi_+^{\vphantom{}}$ of~$(i_+^{\vphantom{}})_*$, then $\phi_+^{\vphantom{}}\circ f_*$ is a splitting of~$(i_+')_*$. Appealing to Proposition~\ref{proposition:splitting}, we complete the proof.
\end{proof}

Thus, the same proof of Theorem~\ref{theorem:infinite rank} shows the following:
\begin{theorem}
  \label{theorem:infinite rank subgroups}
  If $b_1(\Sigma)>1$, then the abelianizations of the subgroups $\B\H$, $\widehat\H$ and $\H(\infty)$ contain a subgroup isomorphic to~$\Z^{\infty}$.
\end{theorem}

\begin{remark}
Since the homology cylinders $E(\alpha, K_i)^2$ mutually commute, they generate an abelian group in $\H(\infty)$. Therefore we obtain that there is an infinite rank free abelian subgroup, say $\mathcal{A}$, of $\B\H$ which injects into the abelianization of any subgroup of $\H(\infty)$ containing $\mathcal{A}$, whenever $b_1(\Sigma)>1$.
\end{remark}

\section{Nilpotent cobordism and solvable cobordism}
  \label{sec:cobordisms}
In this section, we consider other types of cobordisms of homology cylinders, related to gropes and Whitney towers.
Let $M$ and $N$ be homology cylinders over~$\Sigma$.

\begin{definition}
A 4-manifold $W$ which is bounded by $\widehat{M\cdot -N}$ and satisfies $H_1(M)\cong H_1(W)\cong H_1(N)$ under inclusion-induced maps is called a \emph{(relative) $H_1$-cobordism} between $M$ and~$N$.
\end{definition}

For an $H_1$-cobordism $W$ between $M$ and $N$, $H_2(W,M)\cong H_2(W) \cong H_2(W,N)$.
Hence, $W$ is an $H_1$-cobordism with $H_2(W)=0$ if and only if $W$ is a homology cobordism.

As an approximation of homology cobordism, first we define nilpotent cobordism of homology cylinders motivated by \cite{Dw75, FT}. For the definition of closed grope of class $q$, we refer to \cite[Section 2]{FT}.

\begin{definition}
	An $H_1$-cobordism $W$ between $M$ and $N$ is called \emph{class $q$ nilpotent cobordism} if there are maps of closed gropes of class $q$ into $W$, whose base surfaces represent homology classes generating~$H_2(W,N)$. If there exists such $W$, we say that $M$ is \emph{class $q$ nilpotently cobordant} to~$N$.
\end{definition}

There is a relation between nilpotent cobordism and $\tilde\mu$-invariants:

\begin{theorem}
  If $M$ is class $q$ nilpotently cobordant to $N$, then $\tilde\mu_q^{\vphantom{}}(M)=\tilde\mu_q^{\vphantom{}}(N)$.
\end{theorem}

\begin{proof}
  By Dwyer's theorem \cite[Theorem~1.1]{Dw75}), the markings of $M$ and $N$ induce isomorphisms $F/F_q \cong \pi_1(W)/\pi_1(W)_q$.
We consider the commutative diagram below.
  $$\begin{diagram} 
    \node[2] {\pi_1(M)/\pi_1(M)_q} \arrow{se} \\
    \node{F/F_q} \arrow[2]{e,tb}{\qquad(i_+^{\vphantom{}})_{*q}^{\vphantom{}}=(j_+^{\vphantom{}})_{*q}^{\vphantom{}}}{\qquad \cong} \arrow{ne,l}{\cong} \arrow{se,r}{\cong} \node[2]{\pi_1(W)/\pi_1(W)_q} \\
    \node[2]{\pi_1(N)/\pi_1(N)_q} \arrow{ne}
  \end{diagram}$$
Since $\tilde\l(M)_k^{\vphantom{}}$ and $\tilde\l(N)_k^{\vphantom{}}$ are sent to the same element in $\pi_1(W)$ for each~$k$, we obtain $\tilde\mu_q^{\vphantom{}}(M)= \tilde\mu_q^{\vphantom{}}(N)$.
\end{proof}

Next, let us consider solvable cobordism of homology cylinders. 
For a precise definition of solvable cobordism, see \cite[Definition~2.8]{C12}. In \cite{C12}, it is defined between bordered 3-manifolds. Since a homology cylinder is a special case of bordered 3-manifolds, the definition is applied directly to homology cylinders.
A solvable cobordism is also an $H_1$-cobordism approximating homology cobordism. Note that $M$ is $(r)$-solvably cobordant to $N$ if and only if $\widehat{M\cdot(-N)}$ is $(r)$-solvable as a closed 3-manifold (see \cite[Definition 2.1]{COT2}) for $r\in \frac{1}{2}\Z_{\geq 0}$. We say that $M$ is \emph{$(r)$-solvable} if $M$ is $(r)$-solvably cobordant to $E$, or equivalently, if $\hM$ is $(r)$-solvable. 

\begin{theorem}
\label{theorem:solvability}
  Suppose $M \in \H(q)$ and $\T$ is a $p$-structure of height $\leq h$ for $\Sigma$ of order~$q$.
  If either \begin{enumerate}
	\item $M$ is $(h+1)$-solvable or
	\item $M$ is $(h.5)$-solvable and \textnormal{(C1)--(C5)} of Theorem~\ref{theorem:additivity condition} hold for $h+1$,
  \end{enumerate}	
  then $\l_\T(M)$ vanishes.
\end{theorem}

\begin{proof}
  If $M$ satisfies (C2) in Theorem~\ref{theorem:additivity condition} for $h+1$, then $H_1(\Sigma_{(t)};\zp) \cong H_1(\hM_{(t)};\zp)$ for $t\leq h+1$, and so $H_1(\hM_{(t)})$ is $p$-torsion free and $\rank H_1(\Sigma_{(t)};\Q)=\rank H_1(\hM_{(t)};\Q)$ for all $t \leq h+1$. 
The desired conclusion follows immediately from \cite[Theorem 8.2]{C10} and \cite[Theorem 3.2]{C09}.
\end{proof}

By exactly the same argument as in \cite[Proposition~3.1]{COT2}, we have a similar result on homology cylinders infected along a simple closed curve in some derived series:

\begin{lemma}
  \label{lemma:solvable}
	Let $M$ be an $(h)$-solvable homology cylinder. Suppose $\alpha$ is a simple closed curve with $[\alpha]\in (\pi_1(M))^{(h)}$ and $K$ is a knot in $S^3$ with vanishing Arf invariant.
Then, $M(\alpha,K)$ obtained by infection from $M$ along $\alpha$ using $K$ is $(h)$-solvable.
\end{lemma}

We consider the $(r)$-solvable filtration of~$\H$. Denote by $\F^G_{(r)}$ the set of all homology cobordism classes of $(r)$-solvable homology cylinders in~a subgroup $G$ of~$\H$. It can be seen that $\F^G_{(r)}$ is a normal subgroup of $G$ for any subgroup $G$ of~$\H$.
We remark that this may be compared with Kitayama's groups of refined cobordism classes of homology cylinders whose marking induce isomorphisms on solvable quotients \cite{Ki}. 

\begin{theorem}
  \leavevmode \Nopagebreak 
\label{theorem:solvable filtration}
  \begin{enumerate}
	\item
	  For a $p$-structure $\T$ of height $h$ for $\Sigma$, $\l_\T$ gives rise to a homomorphism
	  $$\l_\T \colon \F^{\H(\infty)}_{(h)}/\F^{\H(\infty)}_{(h.5)} \to L^0(\Q(\zeta_d)).$$
	\item The abelianization of $\F^{\H(\infty)}_{(h)}/\F^{\H(\infty)}_{(h.5)}$ 
	 has infinite rank.
  \end{enumerate}
  Both $(1)$ and $(2)$ also hold for $\F^{\B\H}_{(h)}/\F^{\B\H}_{(h.5)}$ and $\F^{\hH}_{(h)}/\F^{\hH}_{(h.5)}$.
\end{theorem}
\begin{proof}
(1) follows from Theorem	~\ref{theorem:solvability} and Corollary~\ref{cor:homomorphism}.
By Lemma~\ref{lemma:solvable}, $E(\alpha,K_i)^2$ in the proof of Theorem~\ref{theorem:infinite rank} is $(h)$-solvable, and is in $\F^{\H(\infty)}_{(h)}$. With the homomorphism in (1), the same argument as in the proof of Theorem~\ref{theorem:infinite rank} proves (2). 
The last sentence follows from that all $E(\alpha,K_i)^2$ are boundary homology cylinders, and also $\hF$-homology cylinders by Lemma~\ref{lemma:preserving}.  \end{proof}


\begin{thebibliography}{99}
\bibitem[CFK11]{CFK} J. C. Cha, S. Friedl, and T. Kim, \emph{The cobordism group of homology cylinders}, Compos. Math. 147 (2011), no. 3, 914--942. 
\bibitem[CH08]{CH08} T. D. Cochran and S. L. Harvey, \emph{Homology and derived series of groups. II. Dwyer's theorem}, Geom. Topol. 12 (2008), no. 1, 199--232.
\bibitem[Cha08]{C08} J. C. Cha, \emph{Injectivity theorems and algebraic closures of groups with coefficients}, Proc. Lond. Math. Soc. (3) 96 (2008), no. 1, 227--250. 
\bibitem[Cha09]{C09} J. C. Cha, \emph{Structure of the string link concordance group and Hirzebruch-type invariants}, Indiana Univ. Math. J. 58 (2009), no. 2, 89--927.
\bibitem[Cha10]{C10} J. C. Cha, \emph{Link concordance, homology cobordism, and Hirzebruch-type defects from iterated $p$-covers}, J. Eur. Math. Soc. (JEMS) 12 (2010), no. 3, 555--610. 
\bibitem[Cha14a]{C} J. C. Cha, \emph{Amenable $L^2$-theoretic methods and knot 
concordance}, Int. Math. Res. Not. IMRN 2014, no. 17, 4768--4803. 
\bibitem[Cha14b]{C12} J. C. Cha, \emph{Symmetric Whitney tower cobordism for bordered 3-manifolds and links}, Trans. Amer. Math. Soc.  366  (2014),  no. 6, 3241--3273.
\bibitem[CHH12]{CHH} T. D. Cochran, S. L. Harvey, and P. D. Horn, \emph{Higher-order signature cocycles for subgroups of mapping class groups and homology cylinders}, Int. Math. Res. Not. IMRN 2012, no. 14, 3311--3373. 
\bibitem[CHL11]{CHL} T. D. Cochran, S. L. Harvey, and C. Leidy, \emph{2-torsion in the n-solvable filtration of the knot concordance group}, Proc. Lond. Math. Soc. (3) 102 (2011), no. 2, 257--290. 
\bibitem[CK08]{CK} T. D. Cochran, T. Kim, \emph{Higher-order Alexander invariants and filtrations of the knot concordance group}, Trans. Amer. Math. Soc. 360 (2008), no. 3, 1407--1441 (electronic). 
\bibitem[CO93]{CO} T. D. Cochran, K. E. Orr, \emph{Not all links are concordant to boundary links}, Ann. Math. (2) 138 (1993), 519--554.
\bibitem[COT03]{COT} T. D. Cochran, K. E. Orr, and P. Teichner, \emph{Knot concordance, Whitney towers, and $L^2$-signatures}, Ann. Math. (2) 157 (2003), no. 2, 433--519.
\bibitem[COT04]{COT2} T. D. Cochran, K. E. Orr, and P. Teichner, \emph{Structure in the classical knot concordance group}, Comment. Math. Helv 79 (2004), 105--123.
\bibitem[CST12]{CST} J. Conant, R. Schneiderman, and P. Teichner, \emph{Geometric filtrations of string links and homology cylinders}, \texttt{arXiv:1202.2482}, to appear in Quantum topology.
\bibitem[Dwy75]{Dw75} W. G. Dwyer, \emph{Homology, Massey products and maps between groups}, J. Pure Appl. Algebra 6, 177--190 (1975)
\bibitem[FM12]{FM} B. Farb and D. Margalit, \emph{A primer on mapping class groups}, Princeton Mathematical Series, 49. Princeton University Press, Princeton, NJ, 2012.
\bibitem[FT95]{FT} M. H. Freedman and P. Teichner, \emph{4-manifold topology. II. Dwyer's filtration and surgery kernels}, Invent. Math. 122 (1995),  no. 3, 531--557. 
\bibitem[GL05]{GL} S. Garoufalidis and J. P. Levine, \emph{Tree-level invariants of three-manifolds, Massey products and the Johnson homomorphism}, 
\bibitem[Gou99]{Go} M. Goussarov, \emph{Finite type invariants and n-equivalence of 3-manifolds}, C. R. Acad. Sci. Paris Sér. I Math. 329 (1999), no. 6, 517--522.
\bibitem[Gru57]{Gr} K. W. Gruenberg, \emph{Residual properties of infinite solvable groups}, Proc. London Math Soc., No. 25, 29--63 (1957)  
\bibitem[GS13]{GS} H. Goda, T. Sakasai, \emph{Homology cylinders and sutured manifolds for homologically fibered knots}, Tokyo J. Math.  36  (2013), no. 1, 85--111.
\bibitem[Hab00]{Ha} K. Habiro, \emph{Claspers and finite type invariants of links}, Geom. Topol. 4 (2000), 1--83 (electronic). 
\bibitem[Har08]{Ha08} S. L. Harvey, \emph{Homology cobordism invariants and the Cochran-Orr-Teichner filtration of the link concordance group}, Geom. Topol. 12 (2008), no. 1, 387--430.
\bibitem[HL90]{HL} N. Habegger and X. S. Lin, \emph{The classification of links up to link-homotopy}, J. Amer. Math. Soc. 3 (1990), no. 2, 389--419.
\bibitem[HL98]{HL98}N. Habegger and X. S. Lin, \emph{On link concordance and Milnor's invariants}, Bull. London Math. Soc. 30, (1998) 419--428.
\bibitem[HM10]{HM10} K. Habiro and G. Massuyeau, \emph{Symplectic Jacobi diagrams and the Lie algebra of homology cylinders}, J. Topol. 2 (2009), no. 3, p. 527--569.
\bibitem[HM12]{HM} K. Habiro and G. Massuyeau, \emph{From mapping class groups to monoids of homology cobordisms: a survey}, Handbook of Teichm\"uller theory. Volume III, 465--529, IRMA Lect. Math. Theor. Phys., 17, Eur. Math. Soc., Z\"urich, 2012. 
\bibitem[Joh80]{Jo} D. Johnson, \emph{An abelian quotient of the mapping class group $\mathcal{I}_g$}, Math. Ann. 249  (1980), no. 3, 225--242
\bibitem[Kit12]{Ki} T. Kitayama, \emph{Homology cylinders of higher-order}, Algebr. Geom. Topol. 12 (2012), no. 3, 1585--1605. 
\bibitem[Lev89]{L89} J. P. Levine, \emph{Link concordance and algebraic closure. {I}{I}}, Invent. Math. 96 (1989), no. 3, 571--592.
\bibitem[Lev90]{L90} J. P. Levine, \emph{Algebraic closure of groups}, Contemp. Math. 109 (1990), 99--105.
\bibitem[Lev94]{L94} J. P. Levine, \emph{Link invariants via the eta invariant}, Comment. Math. Helv. 69 (1994), 82--119.
\bibitem[Lev01]{L01} J. P. Levine, \emph{Homology cylinders: an enlargement of the mapping class group}, Algebr. Geom. Topol. 1 (2001), 243--270 (electronic).
\bibitem[Mas07]{Ma} G. Massuyeau, \emph{Finite-type invariants of 3-manifolds and the dimension subgroup problem}, J. Lond. Math. Soc. (2) 75 (2007), no. 3, p. 791--811
\bibitem[Mil57]{Mi57} J. Milnor \emph{Isotopy of links}, Algebraic geometry and topology.  A symposium in honor of S. Lefschetz, pp. 280--306. Princeton University Press, Princeton, N. J., 1957. 
\bibitem[MKS76]{MKS} W. Magnus, A. Karrass, and D. Solitar \emph{Combinatorial group theory. Presentations of groups in terms of generators and relations}, Second revised edition. Dover Publications, Inc., New York, 1976.
\bibitem[MM10]{MM} G. Massuyeau and J. B. Meilhan, \emph{Equivalence relations for homology cylinders and the core of the Casson invariant}, Trans. Amer. Math. Soc. 365  (2013),  no. 10, 5431--5502.
\bibitem[Mor08]{M} S. Morita, \emph{Symplectic automorphism groups of nilpotent quotients of fundamental groups of surfaces}, Adv. Stud. Pure Math., 52, Math. Soc. Japan, Tokyo, 2008. 
\bibitem[Orr89]{Or} K. E. Orr, \emph{Homotopy invariants of links}, Invent. Math. 95 (1989), no. 2, 379--394.
\bibitem[Sak08]{S08} T. Sakasai, \emph{The Magnus representation and higher-order Alexander invariants for homology cobordisms of surfaces}, Algebr. Geom. Topol. 8 (2008), no. 2, 803--848. 
\bibitem[Sak12]{S} T. Sakasai, \emph{A survey of Magnus representations for mapping class groups and homology cobordisms of surfaces}, Handbook of Teichm\"uller theory. Volume III, 531--594, IRMA Lect. Math. Theor. Phys., 17, Eur. Math. Soc., Z\"urich, 2012.
\bibitem[Sch05]{Sc} R. Schneiderman, \emph{Whitney towers and gropes in 4-manifolds}, Trans. Amer. Math. Soc. 358 (2006), no. 10, 4251--4278 (electronic).
\bibitem[Smy65]{Sm} N. Smythe, \emph{Boundary links}, Topology Seminar, Wisconsin, 1965. Ann. Math. Studies, No. 60, Princeton University Press, Princeton, N.J. 1966, 69--72. 
\bibitem[Sta65]{St} J. Stallings, \emph{Homology and central series of groups}, J. Algebra 2 (1965), 170--181.
\end{thebibliography}
\end{document}